\documentclass[a4paper,12pt,twoside]{article}
\usepackage{amsmath,amsfonts,amssymb,amsthm}
\usepackage{graphicx,color}

\numberwithin{equation}{section}
     \newtheorem{thm}{Theorem}[section]
     \newtheorem{cor}[thm]{Corollary}
     \newtheorem{prop}[thm]{Proposition}
     \newtheorem{lem}[thm]{Lemma}
\theoremstyle{definition}
      \newtheorem{defn}{Definition}[section]
     \newtheorem{exmp}{Example}[section]
\theoremstyle{remark}
     \newtheorem{rem}{Remark}[section]

\newcommand{\R}{\mathbb{R}}

\newcommand{\Z}{\mathbb{Z}}

\newcommand{\cL}{\mathcal{L}}

\newcommand{\cW}{\mathcal{W}}

\newcommand{\supp}{\operatorname{supp}}

\newcommand{\loc}{\mathrm{loc}}
\newcommand{\comp}{\mathrm{comp}}

\newcommand{\BMO}{\mathrm{BMO}}
\newcommand{\CMO}{\mathrm{CMO}}
\newcommand{\CBMO}{\mathrm{CBMO}}
\newcommand{\Lip}{\mathrm{Lip}}
\newcommand{\dB}{\dot{B}}

\newcommand{\tI}{\tilde{I}}
\newcommand{\tT}{\tilde{T}}

\newcommand{\tp}{\tilde{p}}

\newcommand{\tw}{\tilde{w}}

\newcommand{\iTheta}{{\it\Theta}_*}

\newcommand{\dmn}{U}

\newcommand{\Bwu}{B_{w}^{u}(E)}
\newcommand{\dBwu}{\dot{B}_{w}^{u}(E)}
\newcommand{\EQ}{\{E(Q_r)\}}

\newcommand{\ls}{\lesssim}

\newcommand{\sqc}[1]{\{#1_j\}_{j=1}^{\infty}}    

\newcommand{\msckw}{%
\footnotetext{\hspace{-4mm} 
2000 {\it Mathematics Subject Classification}. 
Primary 42B35, 46B70; Secondary 46E30, 46E35, 42B20, 42B25.
\endgraf
{\it Key words and phrases}. 
interpolation,
Morrey spaces, Campanato space, BMO, $B^p$-space, CMO,
Hardy-Littlewood maximal operator,
singular integral operator,
fractional integral operator, 
\endgraf
Eiichi Nakai (Corresponding author):
Department of Mathematics,
Ibaraki University,
Mito, Ibaraki 310-8512, Japan \ \
E-mail: enakai@mx.ibaraki.ac.jp \ Tel: (+081)29-228-8346
\endgraf
Takuya Sobukawa:
Global Education Center,
Waseda University,
Nishi-Waseda, Shinjuku-ku, Tokyo 169-8050, Japan \ \
E-mail: sobu@waseda.jp 
\endgraf
}}

\title{$B_w^u$-function spaces and their interpolation \msckw}
\author{Eiichi Nakai and Takuya Sobukawa}
\date{October 20, 2014}

\begin{document}

\baselineskip=18pt

\maketitle

\begin{abstract}
We introduce $B_w^u$-function spaces which unify 
Lebesgue, Morrey-Campanato, Lipschitz, $B^p$, $\CMO$, 
local Morrey-type spaces,
etc.,
and investigate the interpolation property of $B_w^u$-function spaces.
We also apply it to the boundedness of linear and sublinear operators,
for example, 
the Hardy-Littlewood maximal and fractional maximal operators,
singular and fractional integral operators with rough kernel,
the Littlewood-Paley operator, Marcinkiewicz operator,
and so on.
\end{abstract}

\section{Introduction}\label{s:intro}

The purpose of this paper 
is to introduce $B_w^u$-function spaces which unify many function spaces,
Lebesgue, Morrey-Campanato, Lipschitz, $B^p$, $\CMO$, 
local Morrey-type spaces,
etc.
We investigate the interpolation property of $B_w^u$-function spaces
and apply it to the boundedness of linear and sublinear operators,
for example, the Hardy-Littlewood maximal operator,
singular and fractional integral operators,
and so on,
which contains previous results and extends them to $B_w^u$-function spaces.

Let $\R^n$ be the $n$-dimensional Euclidean space.
We denote by $Q_r$ 
the open cube centered at the origin and sidelength $2r$, or 
the open ball centered at the origin and of radius $r$,
that is, 
\begin{equation*}
 Q_r
 =\left\{y=(y_1,y_2,\cdots,y_n)\in\R^n:\max_{1\le i\le n} |y_i|<r
  \right\}
 \quad\text{or}\quad
 Q_r=\{y\in\R^n:|y|<r\}.
\end{equation*}

For each $r\in(0,\infty)$, 
let $E(Q_r)$ be a function space on $Q_r$ with quasi-norm $\|\cdot\|_{E(Q_r)}$.
Let $E_{Q}(\R^n)$ be the set of all measurable functions $f$ on $\R^n$
such that $f|_{Q_r}\in E(Q_r)$ for all $r>0$.
We assume the following {\it restriction property}\,:
\begin{multline}\label{restriction}
 f|_{Q_r}\in E(Q_r) \ \text{and} \ 0<t<r<\infty
\\ \Rightarrow 
 f|_{Q_t}\in E(Q_t) \ \text{and} \ \|f\|_{E(Q_t)}\le C_E\|f\|_{E(Q_r)},
\end{multline}
where $C_E$ is a positive constant independent of $r$, $t$ and $f$.
For example, $E=L^p$, $\Lip_{\alpha}$, $\BMO$, etc.
Then, for a weight function $w:(0,\infty)\to(0,\infty)$ 
and an exponent $u\in(0,\infty]$, 
we define function spaces $\Bwu=B_{w}^{u}(E)(\R^n)$ and $\dBwu=\dB_{w}^{u}(E)(\R^n)$ 
as the sets of all functions $f\in E_Q(\R^n)$ such that
$\|f\|_{B_{w}^{u}(E)}<\infty$ and $\|f\|_{\dB_{w}^{u}(E)}<\infty$, 
respectively,
where
\begin{align*}
 \|f\|_{B_{w}^{u}(E)}&=\|w(r)\|f\|_{E(Q_r)}\|_{L^{u}([1,\infty),dr/r)},
\\
 \|f\|_{\dB_{w}^{u}(E)}&=\|w(r)\|f\|_{E(Q_r)}\|_{L^{u}((0,\infty),dr/r)}.
\end{align*}
In the above 
we abbreviated $\|f|_{Q_r}\|_{E(Q_r)}$ to $\|f\|_{E(Q_r)}$.

In this paper we always assume that $w$ has some decreasingness condition.
Note that, 
if $w(r)\to\infty$ as $r\to\infty$, then $\Bwu=\dBwu=\{0\}$.
In particular, if $w(r)=r^{-\sigma}$, $\sigma\ge0$ and $u=\infty$, we denote 
$B_{w}^{u}(E)(\R^n)$ and $\dB_{w}^{u}(E)(\R^n)$ 
by
$B_{\sigma}(E)(\R^n)$ and $\dB_{\sigma}(E)(\R^n)$, respectively,
which were introduced recently by
Komori-Furuya, Matsuoka, Nakai and Sawano~\cite{KoMaNaSa2013RMC}.
These $B_{\sigma}$-function spaces unify several function spaces,
see the following Examples~\ref{exmp:Bp}--\ref{exmp:BsMC}.
Moreover, if $E=L^p$, then
$\dB_{w}^{u}(L^p)(\R^n)$ is the local Morrey-type space
introduced by Burenkov and Guliyev~\cite{BuGu2004},
see Example~\ref{exmp:locMt}.

\begin{exmp}\label{exmp:Bp}
Beurling \cite{Beurling1964} introduced the space $B^p(\R^n)$ 
together with its predual $A^p(\R^n)$ so-called the Beurling algebra.
Later, to extend 
Wiener's ideas \cite{Wiener1930,Wiener1932} 
which describe the behavior of functions at infinity, 
Feichtinger \cite{Feicht1987} gave an equivalent norm on $B^p(\R^n)$, 
which is a special case of norms to describe 
non-homogeneous Herz spaces $K^{\alpha}_{p,r}(\R^n)$ 
introduced in \cite{Herz1968}. 
The function space $B^p(\R^n)$ and its homogeneous version $\dB^p(\R^n)$ 
are characterized by the following norms, respectively:
\begin{equation*}
 \|f\|_{B^p}=\sup_{r\ge1}
  \left(\frac1{|Q_r|}\int_{Q_r}|f(x)|^p\,dx\right)^{1/p} 
\ \text{and} \ 
 \|f\|_{\dB^p}=\sup_{r>0}
  \left(\frac1{|Q_r|}\int_{Q_r}|f(x)|^p\,dx\right)^{1/p},
\end{equation*}
where $|Q_r|$ is the Lebesgue measure of $Q_r$.
In this case 
$B^p(\R^n)=B_{\sigma}(L^p)(\R^n)$
and $\dB^p(\R^n)=\dB_{\sigma}(L^p)(\R^n)$ with $\sigma=n/p$.
\end{exmp}

\begin{exmp}\label{exmp:CMO}
Chen and Lau \cite{ChenLau1989} and Garc\'{\i}a-Cuerva \cite{Gar1989} 
introduced the central mean oscillation space $\CMO^{p}(\R^n)$
with the norm
\begin{equation*}
 \|f\|_{\CMO^{p}}=\sup_{r\ge1}
  \left(\frac1{|Q_r|}\int_{Q_r}|f(x)-f_{Q_r}|^p\,dx\right)^{1/p}, 
\end{equation*}
and 
Lu and Yang \cite{LuYang1992Studia,LuYang1995Approx}
introduced the central bounded mean oscillation space $\CBMO^{p}(\R^n)$
with the norm
\begin{equation*}
 \|f\|_{\CBMO^{p}}=\sup_{r>0}
  \left(\frac1{|Q_r|}\int_{Q_r}|f(x)-f_{Q_r}|^p\,dx\right)^{1/p},
\end{equation*}
where $f_{Q_r}$ is the mean value of $f$ on $Q_r$.
Then $\CMO^{p}(\R^n)$ and $\CBMO^{p}(\R^n)$ are expressed by 
$B_{\sigma}(E)(\R^n)$ and $\dB_{\sigma}(E)(\R^n)$, respectively, 
with $E=L^p$ (modulo constants), 
$
 \|f\|_{E(Q_r)}=\|f-f_{Q_r}\|_{L^p(Q_r)}
$
and $\sigma=n/p$.
\end{exmp}

\begin{exmp}\label{exmp:Bpl}
Garc\'{\i}a-Cuerva and Herrero \cite{GarHer1994} and
Alvarez, Guzm\'an-Partida and Lakey \cite{AlvGuzLak2000}
introduced 
the non-homogeneous central Morrey space $B^{p,\lambda}(\R^n)$,
the central Morrey space $\dB^{p,\lambda}(\R^n)$,
the $\lambda$-central mean oscillation space $\CMO^{p,\lambda}(\R^n)$ 
and the $\lambda$-central bounded mean oscillation space $\CBMO^{p,\lambda}(\R^n)$
as an extension of $B^p(\R^n)$, $\dB^p(\R^n)$, $\CMO^p(\R^n)$ and $\CBMO^p(\R^n)$,
respectively, with the following norms:
\begin{align*}
 \|f\|_{B^{p,\lambda}}&=\sup_{r\ge1}
  \frac1{r^{\lambda}}\left(\frac1{|Q_r|}\int_{Q_r}|f(x)|^p\,dx\right)^{1/p}, 
 \\
 \|f\|_{\dB^{p,\lambda}}&=\sup_{r>0}
  \frac1{r^{\lambda}}\left(\frac1{|Q_r|}\int_{Q_r}|f(x)|^p\,dx\right)^{1/p}, 
 \\
 \|f\|_{\CMO^{p,\lambda}}&=\sup_{r\ge1}
  \frac1{r^{\lambda}}\left(\frac1{|Q_r|}\int_{Q_r}|f(x)-f_{Q_r}|^p\,dx\right)^{1/p}
 \ \text{and}
 \\
 \|f\|_{\CBMO^{p,\lambda}}&=\sup_{r>0}
  \frac1{r^{\lambda}}\left(\frac1{|Q_r|}\int_{Q_r}|f(x)-f_{Q_r}|^p\,dx\right)^{1/p}. 
\end{align*}
Then these spaces 
are expressed by $B_{\sigma}(E)(\R^n)$ and $\dB_{\sigma}(E)(\R^n)$
with $E=L^p$ (or $E=L^p$ (modulo constants)) 
and $\sigma=n/p+\lambda$.
\end{exmp}

\begin{exmp}\label{exmp:BsMC}
If
$E=L_{p,\lambda}$ (Morrey space) or $\cL_{p,\lambda}$ (Campanato space),
then the function spaces
$B_{\sigma}(L_{p,\lambda})(\R^n)$, $\dB_{\sigma}(L_{p,\lambda})(\R^n)$,
$B_{\sigma}(\cL_{p,\lambda})(\R^n)$ and $\dB_{\sigma}(\cL_{p,\lambda})(\R^n)$
unify the function spaces in above examples 
and the usual Morrey-Campanato and Lipschitz spaces.
Actually, 
if $\lambda=-n/p$, then $L_{p,\lambda}=L^p$.
If $\sigma=0$, then
$B_{0}(L_{p,\lambda})(\R^n)=\dB_{0}(L_{p,\lambda})(\R^n)
=L_{p,\lambda}(\R^n)$ and
$B_{0}(\cL_{p,\lambda})(\R^n)=\dB_{0}(\cL_{p,\lambda})(\R^n)
=\cL_{p,\lambda}(\R^n)$.
If $\lambda=0$, 
then $\cL_{p,\lambda}(\R^n)=\BMO(\R^n)$ for all $p\in[1,\infty)$
(John and Nirenberg~\cite{JohnNirenberg1961}).
If $\lambda=\alpha\in(0,1]$, 
then $\cL_{p,\lambda}(\R^n)=\Lip_{\alpha}(\R^n)$ for all $p\in[1,\infty)$ 
(Campanato~\cite{Campanato1963}, Meyers~\cite{Meyers1964}, Spanne~\cite{Spanne1965}).
$B_{\sigma}$-Morrey-Campanato spaces were investigated in 
\cite{KomoriMatsu2010Maratea,KoMaNaSa2013RMC,KoMaNaSa2013JFSA,MatsuNakai2011}.
For the definitions of $L_{p,\lambda}$ and $\cL_{p,\lambda}$, 
see Subsection~\ref{ss:MCL}.
\end{exmp}

\begin{exmp}\label{exmp:locMt}
Burenkov and Guliyev~\cite{BuGu2004}
introduced local Morrey-type space $LM_{p\theta,w}(\R^n)$
with the (quasi-)norm
\begin{equation*}
 \|f\|_{LM_{p\theta,w}}
 =
 \|w(r)\|f\|_{L^p(Q_r)}\|_{L^{\theta}(0,\infty)},
\end{equation*}
and investigated the boundedness of 
the Hardy-Littlewood maximal operator.
$LM_{p\theta,\tw}(\R^n)$ is expressed by $\dB_{w}^{u}(E)(\R^n)$
with $E=L^p$ and $\tw(r)=w(r)/r$.
For recent progress of local Morrey-type spaces, 
see \cite{Burenkov2012,Burenkov2013}.
See also \cite{BuDaNu2013,BuNu2009} for interpolation spaces 
for local Morrey-type spaces.
\end{exmp}

In this paper we investigate the interpolation property of $B_w^u$-function spaces
\begin{align*}
 (\dB_{w_0}^{u_0}(E)(\R^n),\dB_{w_1}^{u_1}(E)(\R^n))_{\theta,u}
&=\dB_{w}^{u}(E)(\R^n).
\end{align*}
Moreover, we give the interpolation property with $w=w_0\,\Theta(w_1/w_0)$
for some pseudoconcave function $\Theta$
(Theorem~\ref{thm:IP}).
To do this we assume that,
for any $f\in E_Q(\R^n)$
and for any $r>0$, 
there exists a decomposition $f=f_0^r+f_1^r$ such that
\begin{equation}\label{decomposition0}
 \|f_0^r\|_{E(Q_t)}\le
\begin{cases}
 C_E\|f\|_{E(Q_t)} & (0<t<r), \\
 C_E\|f\|_{E(Q_{ar})} & (r\le t<\infty),
\end{cases}
\end{equation}
and
\begin{equation}\label{decomposition1}
 \|f_1^r\|_{E(Q_t)}\le
\begin{cases}
 0 & (0<t<cr), \\
 C_E\|f\|_{E(Q_{bt})} & (cr\le t<\infty),
\end{cases}
\end{equation}
where $C_E,a,b,c$ are positive constants independent of $r$, $t$ and $f$.
We call the {\it decomposition property} such property.
For example, Lebesgue, Orlicz, Lorentz and Morrey spaces 
have the decomposition property.
Actually, $f=f\chi_r+f(1-\chi_r)$ is the desired decomposition,
where $\chi_{r}$ is the characteristic function of $Q_r$.
Moreover, we prove that Campanato and Lipschitz spaces also have 
the decomposition property
(Proposition~\ref{prop:decomp}).

As applications of the interpolation property,
we also give the boundedness of linear and sublinear operators.
It is known that the Hardy-Littlewood maximal operator,
fractional maximal operators, singular and fractional integral operators
are bounded on $B_{\sigma}$-Morrey-Campanato spaces, 
see \cite{KomoriMatsu2010Maratea,KoMaNaSa2013RMC,KoMaNaSa2013JFSA,MatsuNakai2011}.
Using these boundedness, we get the boundedness of these operators on
$B_w^u(L_{p,\lambda})$,$\dB_w^u(L_{p,\lambda})$,
$B_w^u(\cL_{p,\lambda})$ and $\dB_w^u(\cL_{p,\lambda})$,
which are also generalization of the results 
on the local Morrey-type spaces $LM_{pu,w}(\R^n)$.

We give notation and definitions in Section~\ref{s:def} 
to state main results in Section~\ref{s:main}.
We prove them in Section~\ref{s:proof}
and give applications for the boundedness of linear and sublinear operators
in Section~\ref{s:bdd}.

\section{Notation and definitions}\label{s:def}

In this section 
we give several notation and definitions
to state main result.

A function $w:(0,\infty)\to(0,\infty)$ is said to be 
almost increasing (almost decreasing)
if there exists a constant $C>0$ such that 
\begin{equation}\label{almost} 
     w(r)\le Cw(s) \quad (w(r)\ge Cw(s)) 
     \quad\text{for}\quad r\le s.
\end{equation} 
A function $w:(0,\infty)\to(0,\infty)$ is said to  
satisfy the doubling condition
if there exists a constant $C>0$ such that 
\begin{equation}\label{double} 
     C^{-1}\le\frac{w(r)}{w(s)}\le C 
     \quad\text{for}\quad \frac{1}{2}\le\frac{r}{s}\le 2.
\end{equation} 
For functions $w_1, w_2:(0,\infty)\to(0,\infty)$,
we write $w_1\sim w_2$ 
if there exists a constant $C>0$ such that
\begin{equation}\label{equiv} 
     C^{-1}\le\frac{w_1(r)}{w_2(r)}\le C 
     \quad\text{for}\quad r>0.
\end{equation} 

Note that, if $w_1\sim w_2$, then 
$B_{w_1}^u(E)=B_{w_2}^u(E)$ and 
$\dB_{w_1}^u(E)=\dB_{w_2}^u(E)$
with equivalent norms.
Note also that, if $w$ satisfies the doubling condition,
then,
for any $\eta>0$,
$\|w(r)\|f\|_{E(Q_r)}\|_{L^{u}([\eta,\infty),dr/r)}$
and
$\|w(r)\|f\|_{E(Q_r)}\|_{L^{u}([1,\infty),dr/r)}$
are equivalent each other,
by the restriction property of $\EQ$. 

We denote by $\cW^{u}$, $u\in(0,\infty]$, the set of 
all almost decreasing functions $w:(0,\infty)\to(0,\infty)$
such that $w$ satisfies the doubling condition and
$w\in L^u([1,\infty),dr/r)$.
Note that, if $w\notin L^u([1,\infty),dr/r)$, then
$\Bwu=\dBwu=\{0\}$.
We also denote by $\cW^{*}$ the set of 
all almost decreasing functions $w:(0,\infty)\to(0,\infty)$
such that
$w$ satisfies the doubling condition and 
\begin{equation}\label{int*}
 \int_r^{\infty}w(t)\,\frac{dt}t\le Cw(r),
 \quad r\in(0,\infty),
\end{equation}
where $C$ is a positive constant independent of $r$.
If $w$ satisfies the doubling condition,
then
\begin{equation*}
 w(r)\le C\int_r^{\infty}w(t)\,\frac{dt}t,
 \quad r\in(0,\infty),
\end{equation*}
for some positive constant $C$ independent of $r$,
that is,
the condition \eqref{int*} implies that $w(r)\sim\int_0^rw(t)\,dt/t$.
Then
the condition \eqref{int*} is equivalent that
there exists a positive constant $\epsilon$ such that
$w(r)r^{\epsilon}$ is almost decreasing,
see \cite[Lemma~7.1]{Nakai2008AMSin}.
Therefore, we have the relation
\begin{equation*}
 \cW^*\subset\cW^{u_1}\subset\cW^{u_2}\subset\cW^{\infty},
 \quad 0<u_1<u_2<\infty.
\end{equation*}
Moreover, if $w$ satisfies the doubling condition, 
then there exists a positive constant $\nu$ such that
$w(r)r^{\nu}$ is almost increasing.
Actually, 
take $\nu$ such that $C\le2^{\nu}$, 
here $C$ is the doubling constant in \eqref{double}.
Then, for $r\le s$, 
choosing an integer $k$ such that $2^{k-1}r\le s<2^{k}r$,
we have
\begin{equation*}
 w(r)r^{\nu}
 \le
 C^{k}w(s)r^{\nu}
 \le
 2^{\nu k}w(s)(s/2^{k-1})^{\nu}
 =
 2^{\nu}w(s)s^{\nu}.
\end{equation*}

We say that a function $\Theta:(0,\infty)\to(0,\infty)$
is pseudoconcave if there exists a concave function 
$\tilde\Theta:(0,\infty)\to(0,\infty)$
such that $\Theta\sim\tilde\Theta$.
All pseudoconcave functions satisfy the doubling condition.
Let $\iTheta$ be the set of all functions $\Theta:(0,\infty)\to(0,\infty)$
such that, 
for some constants $C\in(0,\infty)$ and $\epsilon,\epsilon'\in(0,1)$, 
\begin{equation*}
 \frac{\Theta(tr)}{\Theta(r)}
 \le C
 \max(t^{\epsilon},t^{\epsilon'})
 \quad\text{for all}\ r,t\in(0,\infty).
\end{equation*}
Then all functions $\Theta\in\iTheta$ are pseudoconcave,
see \cite{Peetre1968}.
Note that
$\Theta\in\iTheta$
if and only if 
there exist constants $\epsilon,\epsilon'\in(0,1)$ 
such that
$\Theta(r)r^{-\epsilon}$ is almost increasing and that
$\Theta(r)r^{-\epsilon'}$ is almost decreasing.
In this case $\epsilon\le\epsilon'$.

We consider a couple $(A_0,A_1)=(\dB_{w_0}^{u_0}(E),\dB_{w_1}^{u_1}(E))$ or
$(B_{w_0}^{u_0}(E),B_{w_1}^{u_1}(E))$.
For $f\in A_0+A_1$,
let 
\begin{equation*}
 K(r,f;A_0,A_1)
 =
 \inf_{f=f_0+f_1}\left(\|f_0\|_{A_0}+r\|f_0\|_{A_1}\right)
 \quad (0<r<\infty),
\end{equation*}
where the infimum is taken over all decompositions $f=f_0+f_1$ in $A_0+A_1$.
For a pseudoconcave function $\Theta$ and $u\in(0,\infty]$, let
\begin{equation*}
 (A_0,A_1,\Theta)_{u}
 =
 \left\{f\in E_Q(\R^n):
 \|\Theta(r^{-1})K(r,f;A_0,A_1)\|_{L^u((0,\infty),dr/r)}<\infty
 \right\}.
\end{equation*}
We also consider the following:
\begin{equation*}
 (A_0,A_1,\Theta)_{u,\,[1,\infty)}
 =
 \left\{f\in E_Q(\R^n):
 \|\Theta(r^{-1})K(r,f;A_0,A_1)\|_{L^u([1,\infty),dr/r)}<\infty
 \right\}.
\end{equation*}
In particular, for $\Theta(r)=r^{\theta}$, $\theta\in(0,1)$, we denote 
$(A_0,A_1,\Theta)_{u}$ and
$(A_0,A_1,\Theta)_{u,\,[1,\infty)}$
by $(A_0,A_1)_{\theta,u}$ and
$(A_0,A_1)_{\theta,u,\,[1,\infty)}$,
respectively.


\section{Main results}\label{s:main}

In this section we investigate the interpolation properties of
$\dBwu=\dB_{w}^{u}(E)(\R^n)$ and $\Bwu=B_{w}^{u}(E)(\R^n)$,
using the restriction and decomposition properties 
\eqref{restriction}, \eqref{decomposition0} and \eqref{decomposition1}
of $\big\{(E(Q_r),\|\cdot\|_{E(Q_r)})\big\}_{0<r<\infty}$.

\subsection{Interpolation}\label{ss:Interpo}

The main theorem is the following:

\begin{thm}\label{thm:IP}
Assume that a family $\big\{(E(Q_r),\|\cdot\|_{E(Q_r)})\big\}_{0<r<\infty}$ 
has the restriction and decomposition properties.
Let $u_0,u_1,u\in(0,\infty]$, 
$w_0,w_1\in\cW^{\infty}$,
$\Theta\in\iTheta$
and 
\begin{equation*}
 w=w_0\,\Theta(w_1/w_0).
\end{equation*}
For each $i=0,1$,
if $\min(u_i,u)<\infty$, then we assume that $w_i\in\cW^*$.
Assume also that,
for some positive constant $\epsilon$,
$(w_0(r)/w_1(r))r^{-\epsilon}$ is almost increasing,
or,
$(w_1(r)/w_0(r))r^{-\epsilon}$ is almost increasing.
Then
\begin{equation*}
 (\dB_{w_0}^{u_0}(E)(\R^n),\dB_{w_1}^{u_1}(E)(\R^n),\Theta)_{u}
 =
 \dB_{w}^{u}(E)(\R^n),
\end{equation*}
and
\begin{equation*}
 (B_{w_0}^{u_0}(E)(\R^n),B_{w_1}^{u_1}(E)(\R^n),\Theta)_{u,\,[1,\infty)}
 =
 B_{w}^{u}(E)(\R^n).
\end{equation*}
\end{thm}

\begin{rem}
The function $w=w_0\,\Theta(w_1/w_0)$ in Theorem~\ref{thm:IP}
is in $\cW^{\infty}$,
since
the function $R(r,s)=r\,\Theta(s/r)$ is almost increasing
with respect to both $r$ and $s$.
For properties of pseudoconcave functions, see \cite{GusPee1977}.
If $(w_0(r)/w_1(r))r^{-\epsilon}$ is almost increasing,
then 
$w_1(r)r^{\epsilon}$ is almost decreasing,
that is,
$w_1\in\cW^*$.
Similarly, if $(w_1(r)/w_0(r))r^{-\epsilon}$ is almost increasing,
then 
$w_0\in\cW^*$.
\end{rem}

Take $u_0=u_1=\infty$, $w_0(r)=r^{-\sigma_0}$, $w_1(r)=r^{-\sigma_1}$
in Theorem~\ref{thm:IP}.
Then we have the following:

\begin{cor}\label{cor:IP1}
Assume that a family $\big\{(E(Q_r),\|\cdot\|_{E(Q_r)})\big\}_{0<r<\infty}$ 
has the restriction and decomposition properties.
Let $u\in(0,\infty]$, 
$\sigma_0,\sigma_1\in[0,\infty)$ with $\sigma_0\ne\sigma_1$,
$\Theta\in\iTheta$
and 
\begin{equation}\label{w Theta}
 w(r)=r^{-\sigma_0}\,\Theta(r^{\sigma_0-\sigma_1}).
\end{equation}
If $u<\infty$, we assume that $\sigma_0,\sigma_1\in(0,\infty)$.
Then
\begin{equation*}
 (\dB_{\sigma_0}(E)(\R^n),\dB_{\sigma_1}(E)(\R^n),\Theta)_{u}
 =
 \dB_{w}^{u}(E)(\R^n),
\end{equation*}
and
\begin{equation*}
 (B_{\sigma_0}(E)(\R^n),B_{\sigma_1}(E)(\R^n),\Theta)_{u,\,[1,\infty)}
 =
 B_{w}^{u}(E)(\R^n).
\end{equation*}
\end{cor}

\begin{rem}\label{rem:w Theta}
For any $w\in\cW^*$, there exist $\sigma_0,\sigma_1\in[0,\infty)$
and $\Theta\in\iTheta$
such that \eqref{w Theta} holds.
Actually, 
since $w(r)r^{\nu}$ is almost increasing and 
$w(r)r^{\eta}$ is almost decreasing 
for some positive constants $\nu$ and $\eta$ with $\nu>\eta$,
choosing $\sigma_0,\sigma_1\in[0,\infty)$ and $\epsilon,\epsilon'\in(0,1)$ 
such that
\begin{equation}\label{sig1}
 \sigma_0>\sigma_1, \quad
 \epsilon<\epsilon', \quad
 \sigma_0-(\sigma_0-\sigma_1)\epsilon=\nu, \quad
 \sigma_0-(\sigma_0-\sigma_1)\epsilon'=\eta,
\end{equation}
and setting $\Theta$ as
\begin{equation*}
 \Theta(r^{\sigma_0-\sigma_1})
 =
 w(r)r^{\sigma_0},
\end{equation*}
we have
\begin{equation}\label{sig2}
 \Theta(r^{\sigma_0-\sigma_1})r^{(\sigma_0-\sigma_1)(-\epsilon)}
 =
 w(r)r^{\nu},
 \quad
 \Theta(r^{\sigma_0-\sigma_1})r^{(\sigma_0-\sigma_1)(-\epsilon')}
 =
 w(r)r^{\eta}.
\end{equation}
These show that
$\Theta(r)r^{-\epsilon}$ is almost increasing and
$\Theta(r)r^{-\epsilon'}$ is almost decreasing,
that is $\Theta\in\iTheta$.

Conversely, for any $\Theta\in\iTheta$ and $\sigma_0,\sigma_1\in[0,\infty)$
with $\sigma_0>\sigma_1$,
the function $w$ defined by \eqref{w Theta} is in $\cW^*$
by the relations~\eqref{sig1} and \eqref{sig2}.
\end{rem}

\begin{exmp}\label{exmp:w1}
Let $\sigma_0,\sigma_1\in[0,\infty)$, $\sigma_0>\sigma_1$,
$w_0(r)=r^{\sigma_0}$, $w_1(r)=r^{\sigma_1}$, 
$\alpha,\beta\in(0,1)$, 
and let
\begin{equation*}
 w=w_0\,\Theta(w_1/w_0), \quad
 \Theta(r)=\max(r^{\alpha},r^{\beta}).
\end{equation*}
Then 
\begin{equation*}
 w(r)=\max(r^{-(\sigma_0+\alpha(\sigma_1-\sigma_0))},r^{-(\sigma_0+\beta(\sigma_1-\sigma_0))}),
\end{equation*}
and $\Theta\in\iTheta$, since
\begin{equation*}
 \frac{\Theta(tr)}{\Theta(r)}\le\max(t^{\alpha},t^{\beta})
 \quad\text{for all }r,t\in(0,\infty).
\end{equation*}
\end{exmp}

\begin{exmp}\label{exmp:w2}
Let $\cL$ be the set of 
all continuous functions $\ell:(0,\infty)\to(0,\infty)$ 
for which there exists a constant $c\ge1$ such that 
\begin{equation} \label{logdoubling}
  c^{-1}\le\frac{\ell(s)}{\ell(r)}\le c
  \quad\text{whenever}\quad
  \frac12\le\frac{\log s}{\log r}\le2.
\end{equation}
If $\ell\in\cL$, then, for all $\alpha>0$, 
there exists a constant $c_{\alpha}\ge1$ such that 
\begin{equation} \label{ell-alpha}
  c_{\alpha}^{-1}\ell(r)\le\ell(r^{\alpha})\le c_{\alpha}\ell(r)
  \quad\text{for}\quad
  0<r<\infty.
\end{equation}
For other properties on functions $\ell \in \mathcal{L}$,
see \cite[Section~7]{MiNaOhSh2010JMSJ}.
For example, the following function $\ell_{\beta_1,\beta_2}$ is in $\cL$:
\begin{equation*}
\ell_{\beta_1,\beta_2}(r)
=
\begin{cases}
\left(\log \frac{1}{r}\right)^{-\beta_1}&(0<r<e^{-1}),\\
1&(e^{-1} \le r \le e),\\
(\log r)^{\beta_2}&(e<r),
\end{cases}
\quad \beta_1,\beta_2\in(-\infty,\infty).
\end{equation*}
Let $\sigma_0,\sigma_1\in[0,\infty)$, $\sigma_0>\sigma_1$,
$w_0(r)=r^{-\sigma_0}$, $w_1(r)=r^{-\sigma_1}$, 
$\theta\in(0,1)$, 
and let
\begin{equation*}
 w=w_0\,\Theta(w_1/w_0), \quad
 \Theta(r)=r^{\theta}\ell(r), \quad \ell\in\cL.
\end{equation*}
Then $\Theta\in\iTheta$ and
\begin{equation*}
 w(r)\sim r^{-\sigma}\ell(r), \quad
 \sigma=(1-\theta)\sigma_0+\theta\sigma_1.
\end{equation*}
We can take $\ell_{\beta_1,\beta_2}$ as $\ell$.
\end{exmp}

Take $u=\infty$ and $\Theta(r)=r^{\theta}$ in Corollary~\ref{cor:IP1},
Then we have the following:

\begin{cor}\label{cor:IP2}
Assume that a family $\big\{(E(Q_r),\|\cdot\|_{E(Q_r)})\big\}_{0<r<\infty}$ 
has the restriction and decomposition properties.
Let  
$\sigma_0,\sigma_1\in[0,\infty)$ with $\sigma_0\ne\sigma_1$,
$\theta\in(0,1)$
and 
\begin{equation*}
 \sigma=(1-\theta)\sigma_0+\theta\sigma_1.
\end{equation*}
Then
\begin{equation*}
 (\dB_{\sigma_0}(E)(\R^n),\dB_{\sigma_1}(E)(\R^n))_{\theta,\infty}
 =
 \dB_{\sigma}(E)(\R^n),
\end{equation*}
and
\begin{equation*}
 (B_{\sigma_0}(E)(\R^n),B_{\sigma_1}(E)(\R^n))_{\theta,\infty,\,[1,\infty)}
 =
 B_{\sigma}(E)(\R^n).
\end{equation*}
\end{cor}

Let $E=L^p$. 
Then, using Corollaries~\ref{cor:IP1} and \ref{cor:IP2} 
we have the following:

\begin{exmp}\label{exmp:BsLp}
Take $\sigma_0=\sigma\in(0,\infty)$, $\sigma_1=0$
and $\tau=(1-\theta)\sigma$ with $\theta\in(0,1)$
in Corollary~\ref{cor:IP2}.
Then, since $B_0(L^p)(\R^n)=L^p(\R^n)$,
\begin{equation*}
 (\dB_{\sigma}(L^p)(\R^n),L^p(\R^n))_{\theta,\infty}
 =
 \dB_{\tau}(L^p)(\R^n),
\end{equation*}
and
\begin{equation*}
 (B_{\sigma_1}(L^p)(\R^n),L^p(\R^n))_{\theta,\infty,\,[1,\infty)}
 =
 B_{\tau}(L^p)(\R^n).
\end{equation*}
\end{exmp}

\begin{exmp}\label{exmp:BwLp}
Take $u=\infty$, $\sigma_0=\sigma\in(0,\infty)$, $\sigma_1=0$,
$w(r)=r^{-\sigma}\Theta(r^{\sigma})$ with $w\in\cW^*$ and $\Theta\in\iTheta$,
in Corollary~\ref{cor:IP1}.
Then 
\begin{equation*}
 (\dB_{\sigma}(L^p)(\R^n),L^p(\R^n),\Theta)_{\infty}
 =
 \dB_{w}^{\infty}(L^p)(\R^n),
\end{equation*}
and
\begin{equation*}
 (B_{\sigma_1}(L^p)(\R^n),L^p(\R^n),\Theta)_{\infty,\,[1,\infty)}
 =
 B_{w}^{\infty}(L^p)(\R^n).
\end{equation*}
\end{exmp}

\begin{exmp}\label{exmp:locM}
Take $u\in(0,\infty)$, $\sigma_0,\sigma_1\in(0,\infty)$, 
$w(r)=r^{-\sigma_0}\Theta(r^{\sigma_0-\sigma_1})$ with $w\in\cW^*$ and $\Theta\in\iTheta$,
in Corollary~\ref{cor:IP1}.
Then 
\begin{equation*}
 (\dB_{\sigma_0}(L^p)(\R^n),\dB_{\sigma_1}(L^p)(\R^n),\Theta)_{u}
 =
 \dB_{w}^{u}(L^p)(\R^n),
\end{equation*}
and
\begin{equation*}
 (B_{\sigma_0}(L^p)(\R^n),B_{\sigma_1}(L^p)(\R^n),\Theta)_{u,\,[1,\infty)}
 =
 B_{w}^{u}(L^p)(\R^n).
\end{equation*}
In this case $\dB_{w}^{u}(L^p)(\R^n)$ is 
the local Morrey-type space $LM_{pu,\tw}(\R^n)$ 
with $\tw(r)=w(r)/r$.
\end{exmp}

\subsection{Morrey, Campanato and Lipschitz spaces}\label{ss:MCL}
In this subsection, we consider
Morrey, Campanato and Lipschitz spaces
as concrete examples of the function space $E$ 
which does not satisfy the lattice condition \eqref{lattice}.
Let
\begin{equation*}
 Q(x,r)=x+Q_r=\{x+y:y\in Q_r\}.
\end{equation*}
For a measurable set $G \subset \R^n$, 
we denote by $|G|$ and $\chi_{G}$ 
the Lebesgue measure of $G$ and the characteristic function of $G$, 
respectively.
We also abbreviate $\chi_{Q_r}$ to $\chi_r$.

For a function $f\in L^1_{\loc}(\R^n)$ and a measurable set $G \subset \R^n$ 
with $|G|>0$, let 
\begin{equation}\label{mean}
  f_{G}=\frac1{|G|}\int_{G} f(y)\,dy.
\end{equation}
For a measurable function $f$ on $\R^n$, a measurable set $G \subset \R^n$ 
with $|G|>0$ and $t\in[0,\infty)$, let 
\begin{equation}\label{f>t}
  m(G,f,t)=|\{y\in G: |f(y)|>t\}|.
\end{equation}

We recall the definitions of 
Morrey, weak Morrey, Campanato and Lipschitz spaces below.
These function spaces have the restriction properties.
The first two have also the support property~\eqref{support}
and the lattice property~\eqref{lattice}, and then 
the decomposition property.
The last two also have the decomposition property by
Theorem~\ref{thm:Lip} and Proposition~\ref{prop:decomp}.
Therefore,
we can take these function spaces as $E$ in
Theorem~\ref{thm:IP} and Corollaries~\ref{cor:IP1} and \ref{cor:IP2}.

\begin{defn}\label{defn:MC}
Let $\dmn=\R^n$ or $\dmn=Q_r$ with $r>0$.
For $p\in[1,\infty)$, $\lambda\in\R$ and $\alpha\in(0,1]$, 
let
$L_{p,\lambda}(\dmn)$, $WL_{p,\lambda}(\dmn)$, $\cL_{p,\lambda}(\dmn)$ and
$\Lip_{\alpha}(\dmn)$ be the sets of all functions $f$ such that the 
following functionals are finite, respectively: 
\begin{align*}
\|f\|_{L_{p,\lambda}(\dmn)}
  &=\sup_{Q(x,s)\subset\dmn}\frac{1}{s^{\lambda}}
    \left(\frac1{|Q(x,s)|}\int_{Q(x,s)}|f(y)|^p\,dy\right)^{1/p}, \\
\|f\|_{WL_{p,\lambda}(\dmn)}
  &=\sup_{Q(x,s)\subset\dmn}\frac{1}{s^{\lambda}}
    \left(\frac{\sup_{t>0} t^p\,m(Q(x,s),f,t)}{|Q(x,s)|}\right)^{1/p}, \\
\|f\|_{\cL_{p,\lambda}(\dmn)}
  &=\sup_{Q(x,s)\subset\dmn}\frac{1}{s^{\lambda}}
   \left(\frac1{|Q(x,s)|}\int_{Q(x,s)}|f(y)-f_{Q(x,s)}|^p\,dy\right)^{1/p}, \\
\intertext{and}
\|f\|_{\Lip_{\alpha}(\dmn)}
  &=\sup_{x,y\in\dmn,\,x\ne y}\frac{|f(x)-f(y)|}{|x-y|^{\alpha}}.
\end{align*}
\end{defn}

Then $L_{p,\lambda}(\dmn)$ is a Banach space
and $WL_{p,\lambda}(\dmn)$ is a complete quasi-normed space.
In this paper 
we regard $\cL_{p,\lambda}(\dmn)$ and $\Lip_{\alpha}(\dmn)$ 
as spaces of functions modulo constant functions. 
Then $\cL_{p,\lambda}(\R^n)$ and $\Lip_{\alpha}(\R^n)$
are Banach spaces equipped with the norms 
$\|f\|_{\cL_{p,\lambda}}$ and $\|f\|_{\Lip_{\alpha}}$, respectively.

By the definition, 
if $\lambda=-n/p$, then 
$L_{p,-n/p}(\dmn)=L^p(\dmn)$ and
$WL_{p,-n/p}(\dmn)=WL^p(\dmn)$, the weak $L^p$ space.
If $p=1$ and $\lambda=0$, 
then $\cL_{1,0}(\dmn)$ is the usual $\BMO(\dmn)$.

\begin{rem}\label{rem:Bs(E)}
We note that $B_{\sigma}(L_{p,\lambda})(\R^n)$ unifies $L_{p,\lambda}(\R^n)$ and $B^{p,\lambda}(\R^n)$
and that  $B_{\sigma}(\cL_{p,\lambda})(\R^n)$ unifies $\cL_{p,\lambda}(\R^n)$ and $\CMO^{p,\lambda}(\R^n)$. 
Actually, we have the following relations:
\begin{gather} 
  B_{0}(L_{p,\lambda})(\R^n)=L_{p,\lambda}(\R^n),
  \quad
  B_{0}(\cL_{p,\lambda})(\R^n)=\cL_{p,\lambda}(\R^n), \label{b0} \\
  B_{\lambda+n/p}(L_{p,-n/p})(\R^n)=B^{p,\lambda}(\R^n),
  \quad
  B_{\lambda+n/p}(\cL_{p,-n/p})(\R^n)=\CMO^{p,\lambda}(\R^n).  \label{bl}
\end{gather}
In the above relations, the first three follow immediately from their definitions,
and the last one follows from Theorem~\ref{thm:Morrey} below.
We also have the same properties for the function spaces
$\dB_{\sigma}(L_{p,\lambda})(\R^n)$ and 
$\dB_{\sigma}(\cL_{p,\lambda})(\R^n)$. 
\end{rem}

Here 
we state two known theorems which give the relations 
among Morrey, Campanato and Lipschitz spaces. 
For the proofs of Theorems~\ref{thm:Lip} and \ref{thm:Morrey} below, 
see \cite{Campanato1963,Meyers1964,Spanne1965} and \cite{Mizuhara1995,Nakai2006Studia}, respectively.
For other relations among function spaces in Remark~\ref{rem:Bs(E)},
see \cite[Proposition~1]{KoMaNaSa2013RMC}.

\begin{thm}\label{thm:Lip}
If $p\in[1,\infty)$ and $\lambda=\alpha\in(0,1]$, 
then, for each $r>0$, $\cL_{p,\lambda}(Q_r)=\Lip_{\alpha}(Q_r)$ modulo null-functions
and there exists a positive constant $C$, dependent only on $n$ and $\lambda$, such that
$$
  C^{-1}\|f\|_{\cL_{p,\lambda}(Q_r)}\le\|f\|_{\Lip_{\alpha}(Q_r)}\le C\|f\|_{\cL_{p,\lambda}(Q_r)}.
$$
The same conclusion holds on $\R^n$.
\end{thm}

\begin{thm}\label{thm:Morrey}
If $p\in[1,\infty)$ and $\lambda\in[-n/p,0)$, 
then, for each $r>0$, $\cL_{p,\lambda}(Q_r)\cong L_{p,\lambda}(Q_r)$.
More precisely, the map $f\mapsto f-f_{Q_r}$ is bijective and bicontinuous 
from $\cL_{p,\lambda}(Q_r)$ to $L_{p,\lambda}(Q_r)$, 
that is,
there exists a positive constant $C$, dependent only on $n$ and $\lambda$, such that
$$
 C^{-1}\|f\|_{\cL_{p,\lambda}(Q_r)}
 \le
 \|f-f_{Q_r}\|_{L_{p,\lambda}(Q_r)}
 \le
 C\|f\|_{\cL_{p,\lambda}(Q_r)}.
$$
The same conclusion holds on $\R^n$ by using $\lim_{r\to\infty}f_{Q_r}$ instead of $f_{Q_r}$.
\end{thm}


Now we consider the decomposition property.
Recall that
$E_{Q}(\R^n)$ is the set of all measurable functions $f$ on $\R^n$
such that $f|_{Q_r}\in E(Q_r)$ for all $r>0$.
If the family $\EQ$ has the restriction property
and the following two conditions,
then it has the decomposition property.
\begin{gather}\label{support}
 f\in E(Q_t), \ 0<r<t<\infty \ \text{and} \ \supp f \subset Q_r
 \ \Rightarrow \
 \|f\|_{E(Q_t)}\le C_E\|f\|_{E(Q_r)},
\\
\begin{split}
 g\in E(Q_r) \ \text{and} \ |f(x)|\le|g(x)| &\ \text{for a.e.\,$x\in Q_r$}
\\
 &\Rightarrow \
 f\in E(Q_r) \ \text{and} \ \|f\|_{E(Q_r)}\le C_E\|g\|_{E(Q_r)}.
\end{split}\label{lattice}
\end{gather}
Actually, 
for $f\in E_Q(\R^n)$, letting
\begin{equation*}
 f_0^r=f\chi_{r}, 
 \quad
 f_1^r=f-f_0^r, 
\end{equation*}
we have the desired decomposition with $a=b=c=1$,
where $\chi_{r}$ is the characteristic function of $Q_r$.
Lebesgue, Orlicz and Lorentz spaces satisfy these conditions.
Moreover, Morrey and weak Morrey spaces also satisfy them.

Next we prove the decomposition property of Campanato spaces.
For $r>0$, let 
\begin{equation}\label{h}
  h_r(x)=h(x/r),
  \quad
  h(x)=\begin{cases}1,&|x|\le1,\\0,&|x|\ge2,\end{cases}
  \quad
  \|h\|_{\Lip_1(\R^n)}\le1.
\end{equation}

\begin{prop}\label{prop:decomp}
Let $p\in[1,\infty)$ and $\lambda\in[-n/p,1]$.
Then the family $\{\cL_{p,\lambda}(Q_r)\}$ has the decomposition property.
More precisely, 
for any $f\in(\cL_{p,\lambda})_Q(\R^n)$ and for any $r>0$,
let
\begin{equation*}
 f_0^r=(f-f_{Q_{2r}})h_{r}, \quad
 f_1^r=f-(f-f_{Q_{2r}})h_{r}.  
\end{equation*}
Then $f=f_0^r+f_1^r$,
\begin{equation*}
 \|f_0^r\|_{\cL_{p,\lambda}(Q_t)}\le
\begin{cases}
 C\|f\|_{\cL_{p,\lambda}(Q_t)} & (0<t<r) \\
 C\|f\|_{\cL_{p,\lambda}(Q_{3r})} & (r\le t<\infty),
\end{cases}
\end{equation*}
and
\begin{equation*}
 \|f_1^r\|_{\cL_{p,\lambda}(Q_t)}\le
\begin{cases}
 0 & (0<t<r) \\
 C\|f\|_{\cL_{p,\lambda}(Q_{3t})} & (r\le t<\infty),
\end{cases}
\end{equation*}
where $C$ is a positive constant independent of $r$, $t$ and $f$.
\end{prop}

\begin{proof}
If $0<t<r$, then 
$f_0^r=f-f_{Q_{2r}}$,  
$f_1^r=f_{Q_{2r}}$
and 
\begin{equation*}
 \|f_0^r\|_{\cL_{p,\lambda(Q_t)}}
 =
 \|f\|_{\cL_{p,\lambda(Q_t)}},
 \quad
 \|f_1^r\|_{\cL_{p,\lambda(Q_t)}}
 =
 0.
\end{equation*}
If $r\le t<\infty$, then, 
by the same argument as \cite[Lemma~3.5]{MatsuNakai2011} we have
\begin{equation*}
 \|f_0^r\|_{\cL_{p,\lambda(Q_t)}}
 \le C
 \|f\|_{\cL_{p,\lambda}(Q_{3r})},
\end{equation*}
and
\begin{equation*}
 \|f_1^r\|_{\cL_{p,\lambda(Q_t)}}
 \le 
 \|f\|_{\cL_{p,\lambda(Q_t)}}+\|f_0^r\|_{\cL_{p,\lambda(Q_t)}}
 \le 
 \|f\|_{\cL_{p,\lambda}(Q_{t})}
 +C\|f\|_{\cL_{p,\lambda}(Q_{3r})}
 \le 
 C\|f\|_{\cL_{p,\lambda}(Q_{3t})}.
\end{equation*}
Then we have the conclusion.
\end{proof}

By Theorem~\ref{thm:Lip} we have the following:

\begin{cor}\label{cor:decomp}
Let $\alpha\in(0,1]$.
Then the family $\{\Lip_{\alpha}(Q_r)\}$ has the decomposition property.
\end{cor}

Therefore, it turned out that we can take 
$L_{p,\lambda}$, $WL_{p,\lambda}$, $\cL_{p,\lambda}$, $\BMO$ 
and $\Lip_{\alpha}$
instead of $L^p$ in Examples~\ref{exmp:BsLp}, \ref{exmp:BwLp} and 
\ref{exmp:locM}.
Actually, we have the following:

\begin{exmp}\label{exmp:MC1}
Take $\sigma_0=\sigma\in(0,\infty)$, $\sigma_1=0$
and $\tau=(1-\theta)\sigma$ with $\theta\in(0,1)$
in Corollary~\ref{cor:IP2}.
Then
\begin{equation*}
 (\dB_{\sigma}(E)(\R^n),E(\R^n))_{\theta,\infty}
 =
 \dB_{\tau}(E)(\R^n),
\end{equation*}
and
\begin{equation*}
 (B_{\sigma_1}(E)(\R^n),E(\R^n))_{\theta,\infty,\,[1,\infty)}
 =
 B_{\tau}(E)(\R^n),
\end{equation*}
where 
$E=L_{p,\lambda}$, $WL_{p,\lambda}$, $\cL_{p,\mu}$, $\BMO$, 
or $\Lip_{\alpha}$,
with $p\in[1,\infty)$, $\lambda\in[-n/p,0]$, $\mu\in[-n/p,1]$ 
and $\alpha\in(0,1]$.
\end{exmp}

\begin{exmp}\label{exmp:MC2}
Take $u=\infty$, $\sigma_0=\sigma\in(0,\infty)$, $\sigma_1=0$,
$w(r)=r^{-\sigma}\Theta(r^{\sigma})$ with $w\in\cW^*$ 
and $\Theta\in\iTheta$,
in Corollary~\ref{cor:IP1}.
Then 
\begin{equation*}
 (\dB_{\sigma}(E)(\R^n),E(\R^n),\Theta)_{\infty}
 =
 \dB_{w}^{\infty}(E)(\R^n),
\end{equation*}
and
\begin{equation*}
 (B_{\sigma_1}(E)(\R^n),E(\R^n),\Theta)_{\infty,\,[1,\infty)}
 =
 B_{w}^{\infty}(E)(\R^n),
\end{equation*}
where 
$E=L_{p,\lambda}$, $WL_{p,\lambda}$, $\cL_{p,\mu}$, $\BMO$, 
or $\Lip_{\alpha}$,
with $p\in[1,\infty)$, $\lambda\in[-n/p,0]$, $\mu\in[-n/p,1]$ and $\alpha\in(0,1]$.
\end{exmp}

\begin{exmp}\label{exmp:MC3}
Take $u\in(0,\infty)$, $\sigma_0,\sigma_1\in(0,\infty)$, 
$w(r)=r^{-\sigma_0}\Theta(r^{\sigma_0-\sigma_1})$ with $w\in\cW^*$ and $\Theta\in\iTheta$,
in Corollary~\ref{cor:IP1}.
Then 
\begin{equation*}
 (\dB_{\sigma_0}(E)(\R^n),\dB_{\sigma_1}(E)(\R^n),\Theta)_{u}
 =
 \dB_{w}^{u}(E)(\R^n),
\end{equation*}
and
\begin{equation*}
 (B_{\sigma_0}(E)(\R^n),B_{\sigma_1}(E)(\R^n),\Theta)_{u,\,[1,\infty)}
 =
 B_{w}^{u}(E)(\R^n),
\end{equation*}
where 
$E=L_{p,\lambda}$, $WL_{p,\lambda}$, $\cL_{p,\mu}$, $\BMO$, 
or $\Lip_{\alpha}$,
with $p\in[1,\infty)$, $\lambda\in[-n/p,0]$, $\mu\in[-n/p,1]$ and $\alpha\in(0,1]$.
\end{exmp}

\begin{exmp}\label{exmp:MC4}
Let $p\in[1,\infty)$, $\lambda_0,\lambda_1\in[-n/p,\infty)$, $\theta\in(0,1)$ and
$\lambda=(1-\theta)\lambda_0+\theta\lambda_1$.
Then
\begin{align*}
 (\dB^{p,\lambda_0}(\R^n),\dB^{p,\lambda_1}(\R^n))_{\theta,\infty}
 &=
 \dB^{p,\lambda}(\R^n),
\\
 (\CBMO^{p,\lambda_0}(\R^n),\CBMO^{p,\lambda_1}(\R^n))_{\theta,\infty}
 &=
 \CBMO^{p,\lambda}(\R^n),
\end{align*}
and
\begin{align*}
 (B^{p,\lambda_0}(\R^n),B^{p,\lambda_1}(\R^n))_{\theta,\infty,\,[1,\infty)}
 &=
 B^{p,\lambda}(\R^n),
\\
 (\CMO^{p,\lambda_0}(\R^n),\CMO^{p,\lambda_1}(\R^n))_{\theta,\infty,\,[1,\infty)}
 &=
 \CMO^{p,\lambda}(\R^n).
\end{align*}
\end{exmp}

\section{Proof of the main theorem}\label{s:proof}
To prove the main theorem 
we need several lemmas.
We also use a weighted Hardy's inequality 
by Muckenhoupt~\cite{Muckenhoupt1972}.

\begin{lem}\label{lem:incl}
Let $0<u_0<u_1\le\infty$
and $w:(0,\infty)\to(0,\infty)$.
If $w$ satisfies the doubling condition,
then
\begin{equation*}
 B_{w}^{u_0}(E)(\R^n)
 \subset
 B_{w}^{u_1}(E)(\R^n)
 \quad\text{and}\quad
 \dB_{w}^{u_0}(E)(\R^n)
 \subset
 \dB_{w}^{u_1}(E)(\R^n)
\end{equation*}
with
\begin{equation*}
 \|f\|_{B_{w}^{u_1}(E)}
 \le C
 \|f\|_{B_{w}^{u_0}(E)}
 \quad\text{and}\quad
 \|f\|_{\dB_{w}^{u_1}(E)}
 \le C
 \|f\|_{\dB_{w}^{u_0}(E)},
\end{equation*}
respectively, where $C$ is independent of $f$.
\end{lem}

\begin{proof}
Let $f\in\dB_{w}^{u_1}(E)$.
\begin{align*}
 \|f\|_{\dB_{w}^{u_1}(E)}
 &=\|w(r)\|f\|_{E(Q_r)}\|_{L^{u_1}((0,\infty),dr/r)}
\\
 &=
 \left\|
 \bigg\{\|w(r)\|f\|_{E(Q_r)}\|_{L^{u_1}([2^{j-1},2^j),dr/r)}
 \bigg\}_{j\in\Z}\right\|_{\ell^{u_1}}
\\
 &\ls
 \left\|\bigg\{w(2^j)\|f\|_{E(Q_{2^j})}\bigg\}_{j\in\Z}\right\|_{\ell^{u_1}}
\\
 &\le
 \left\|\bigg\{w(2^j)\|f\|_{E(Q_{2^j})}\bigg\}_{j\in\Z}\right\|_{\ell^{u_0}}
\\
 &\ls
 \left\|\bigg\{
 \|w(r)\|f\|_{E(Q_r)}\|_{L^{u_0}([2^j,2^{j+1}),dr/r)}
 \bigg\}_{j\in\Z}\right\|_{\ell^{u_0}}
\\
 &=\|w(r)\|f\|_{E(Q_r)}\|_{L^{u_0}((0,\infty),dr/r)}
 =\|f\|_{\dB_{w}^{u_0}(E)}.
\end{align*}
For $f\in B_{w}^{u_1}(E)$, take $j\ge1$ instead of $j\in\Z$ in the above
calculation.
\end{proof}

\begin{lem}\label{lem:phi}
Let functions $\phi,G:(0,\infty)\to(0,\infty)$ satisfy the doubling condition,
$\epsilon>0$ and $u\in(0,\infty]$. 
Assume that $\phi(r)r^{-\epsilon}$ is almost increasing or 
$\phi(r)r^{\epsilon}$ is almost decreasing.
Then 
\begin{equation*}
 C^{-1}\|G\|_{L^u((0,\infty),dr/r)}
 \le
 \|G\circ\phi\|_{L^u((0,\infty),dr/r)}
 \le
 C\|G\|_{L^u((0,\infty),dr/r)},
\end{equation*}
and
\begin{equation*}
 C^{-1}\|G\|_{L^u([1,\infty),dr/r)}
 \le
 \|G\circ\phi\|_{L^u([1,\infty),dr/r)}
 \le
 C\|G\|_{L^u([1,\infty),dr/r)},
\end{equation*}
where $C$ is a positive constant depending only on $\epsilon$, $u$ and
the doubling constants of $\phi$ and $G$.
\end{lem}

\begin{proof}
If $\phi$ satisfies the doubling condition and 
$\phi(r)r^{-\epsilon}$ is almost increasing,
then $\phi(r)\sim\int_0^r\phi(t)\,dt/t$.
Let $\phi_1(r)=\int_0^r\phi(t)\,dt/t$. 
Then $\phi_1$ is continuous and $\phi\sim\phi_1$, 
that is, $\phi_1$ satisfies the doubling condition and 
$\phi_1(r)r^{-\epsilon}$ is almost increasing.
Let $\phi_2(r)=\int_0^r\phi_1(t)\,dt/t$.
Then $\phi_2$ is differentiable, strictly increasing and $\phi\sim\phi_2$.
In this case $\phi_2(r)r^{-\epsilon}$ is almost increasing,
and then
$\lim_{r\to0}\phi_2(r)=0$ and
$\lim_{r\to\infty}\phi_2(r)=\infty$.
Therefore, $\phi_2$ is bijective from $(0,\infty)$ to itself.
Moreover,
\begin{equation*}
 \frac{\phi_2'(r)}{\phi_2(r)}
 =\frac{\phi_1(r)/r}{\phi_2(r)}
 \sim\frac{1}{r}.
\end{equation*}
Using the doubling condition of $G$,
we have
\begin{align*}
 \|G\circ\phi\|_{L^u((0,\infty),dr/r)}
 &\sim
 \|G\circ\phi_2\|_{L^u((0,\infty),dr/r)}
\\
 &\sim
 \|G\circ\phi_2\|_{L^u((0,\infty),(\phi_2'(r)/\phi_2(r))dr)}
\\
 &=
 \|G\|_{L^u((0,\infty),dr/r)}.
\end{align*}
Further, let $\phi_3(r)=\phi_2(r)/\phi_2(1)$.
Then $\phi_3(1)=1$ and $\phi_3$ has the same properties as $\phi_2$.
Hence, using $\phi_3$, we have
\begin{equation*}
 \|G\circ\phi\|_{L^u([1,\infty),dr/r)}
 \sim
 \|G\|_{L^u([1,\infty),dr/r)}.
\end{equation*}

If $\phi(r)r^{\epsilon}$ is almost decreasing,
letting $\phi_1(r)=\int_r^{\infty}\phi(t)\,dt/t$
and $\phi_2(r)=\int_r^{\infty}\phi_1(t)\,dt/t$,
we see that
$\phi_2$ is differentiable and bijective from $(0,\infty)$ to itself,
and
\begin{equation*}
 \lim_{r\to0}\phi_2(r)=\infty, \quad \lim_{r\to\infty}\phi_2(r)=0, \quad
 -\frac{\phi_2'(r)}{\phi_2(r)}
 =\frac{\phi_1(r)/r}{\phi_2(r)}
 \sim\frac{1}{r}.
\end{equation*}
In this case, we also have the same conclusion.
\end{proof}

\begin{thm}[{Muckenhoupt \cite{Muckenhoupt1972}}]\label{thm:Muck}
Let $p\in[1,\infty]$. 
Let $F^*(r)=\int_0^rf(t)\,dt$ and $F_*(r)=\int_r^{\infty}f(t)\,dt$.
Then
\begin{equation*}
 \|UF^*\|_{L^p(0,\infty)}\le C\|Vf\|_{L^p(0,\infty)}
\end{equation*}
if and only if
\begin{equation*}
 \sup_{r>0}
 \left(\int_r^{\infty}|U(t)|^p\,dt\right)^{1/p}
 \left(\int_0^r|V(t)|^{-p'}\,dt\right)^{1/p'}<\infty.
\end{equation*}
Also,
\begin{equation*}
 \|UF_*\|_{L^p(0,\infty)}\le C\|Vf\|_{L^p(0,\infty)}
\end{equation*}
if and only if
\begin{equation*}
 \sup_{r>0}
 \left(\int_0^r|U(t)|^p\,dt\right)^{1/p}
 \left(\int_r^{\infty}|V(t)|^{-p'}\,dt\right)^{1/p'}<\infty.
\end{equation*}
\end{thm}

\begin{lem}\label{lem:w*}
Let $u_0,u_1,u\in(0,\infty]$, $\max(u_0,u_1)\le u$,
$w_0,w_1\in\cW^{\infty}$, $\Theta\in\iTheta$,
and let
\begin{equation*}
 w=w_0\,\Theta(w_1/w_0), \quad w_*=w_0/w_1.
\end{equation*}
\begin{enumerate}
\item 
Let $\max(u_0,u_1)<\infty$ and $w_0,w_1\in\cW^*$.
Assume that 
$w_*(r)r^{-\epsilon}$ is almost increasing 
for some positive constant $\epsilon$.
For $f\in\dBwu$, let
\begin{equation*}
 F_0(t)=w_0(t)^{u_0}\|f\|_{E(Q_t)}^{u_0}t^{-1}, \quad
 U_0(r)=\bigg(\Theta(w_*(r)^{-1})\bigg)^{u_0}r^{-u_0/u}, 
\end{equation*}
and
\begin{equation*}
 F_1(t)=w_1(t)^{u_1}\|f\|_{E(Q_t)}^{u_1}t^{-1}, \quad
 U_1(r)=\bigg(w_*(r)\Theta(w_*(r)^{-1})\bigg)^{u_1}r^{-u_1/u}.
\end{equation*}
Then
\begin{equation*}
 \left\|U_0(r)\int_0^r F_0(t)\,dt\right\|_{L^{u/u_0}(0,\infty)}^{1/u_0}
 +\left\|U_1(r)\int_r^{\infty}F_1(t)\,dt\right\|_{L^{u/u_1}(0,\infty)}^{1/u_1}
 \le C
 \|f\|_{\dBwu},
\end{equation*}
where $C$ is independent of $f$.
\item 
Let $u_0=u=\infty$.
Assume that 
$w_*(r)$ is almost increasing. 
For $f\in\dBwu$, let
\begin{equation*}
 F_0(t)=w_0(t)\|f\|_{E(Q_t)}, \quad
 U_0(r)=\Theta(w_*(r)^{-1}). \quad
\end{equation*}
Then
\begin{equation*}
 \bigg\|U_0(r)
  \left(\sup_{t\in(0,r)}F_0(t)\right)
 \bigg\|_{L^{\infty}(0,\infty)}
 \le C
 \|f\|_{\dBwu},
\end{equation*}
where $C$ is independent of $f$.
\item 
Let $u_1=u=\infty$.
Assume that 
$w_*(r)$ is almost increasing. 
For $f\in\dBwu$, let
\begin{equation*}
 F_1(t)=w_1(t)\|f\|_{E(Q_t)}, \quad
 U_1(r)=w_*(r)\Theta(w_*(r)^{-1}). \quad
\end{equation*}
Then
\begin{equation*}
 \bigg\|U_1(r)
  \left(\sup_{t\in(r,\infty)}F_1(t)\right)
 \bigg\|_{L^{\infty}(0,\infty)}
 \le C
 \|f\|_{\dBwu},
\end{equation*}
where $C$ is independent of $f$.
\end{enumerate}
\end{lem}

\begin{rem}\label{rem:w*}
In the definition of $F_0$ and $F_1$ of Lemma~\ref{lem:w*},
using $\|f\|_{E(Q_r)}\chi_{[1,\infty)}(r)$ 
instead of $\|f\|_{E(Q_r)}$,
we have the result for $f\in\Bwu$.
\end{rem}

\begin{proof}[Proof of Lemma~\ref{lem:w*}]
(i)
We may assume that 
$w_*(r)r^{-\epsilon}$ and $\Theta(r)r^{-\epsilon}$ are almost increasing 
and $\Theta(r)r^{\epsilon-1}$ is almost decreasing
for the same small $\epsilon$.
First note that, 
using these properties and the doubling condition of $\Theta$,
we have that, 
for $a>0$,
\begin{align*}
 \int_0^r \bigg(\Theta(w_*(t)^{-1})\bigg)^{-a}\,\frac{dt}t
 &=
 \int_0^r 
  \bigg(\Theta(w_*(t)^{-1})w_*(t)^{\epsilon}\bigg)^{-a}
  \bigg(w_*(t)^{-1}t^{\epsilon}\bigg)^{-\epsilon a}
  \,t^{\epsilon^2 a}\,\frac{dt}t
\\
 &\ls
  \bigg(\Theta(w_*(r)^{-1})w_*(r)^{\epsilon}\bigg)^{-a}
  \bigg(w_*(r)^{-1}r^{\epsilon}\bigg)^{-\epsilon a}
  \int_0^r t^{\epsilon^2 a}\,\frac{dt}t
\\
 &\sim
 \bigg(\Theta(w_*(r)^{-1})\bigg)^a,
\end{align*}
and
\begin{align*}
 \int_0^r \bigg(w_*(t)\Theta(w_*(t)^{-1})\bigg)^a\,\frac{dt}t
 &=
 \int_0^r 
  \bigg(w_*(t)^{1-\epsilon}\Theta(w_*(t)^{-1})\bigg)^a
  \bigg(w_*(t)t^{-\epsilon}\bigg)^{\epsilon a}
  \,t^{\epsilon^2 a}\,\frac{dt}t
\\
 &\ls
 \bigg(w_*(r)^{1-\epsilon}\Theta(w_*(r)^{-1})\bigg)^a
  \bigg(w_*(r)r^{-\epsilon}\bigg)^{\epsilon a}
  \int_0^r t^{\epsilon^2 a}\,\frac{dt}t
\\
 &\sim
 \bigg(w_*(r)\Theta(w_*(r)^{-1})\bigg)^a.
\end{align*}
Similarly, we can get
\begin{align*}
 \int_r^{\infty} \bigg(\Theta(w_*(t)^{-1})\bigg)^a\,\frac{dt}t
 \ls \bigg(\Theta(w_*(r)^{-1})\bigg)^a,
\end{align*}
and
\begin{align*}
 \int_r^{\infty} \bigg(w_*(t)\Theta(w_*(t)^{-1})\bigg)^{-a}\,\frac{dt}t
 \ls \bigg(w_*(r)\Theta(w_*(r)^{-1})\bigg)^{-a}.
\end{align*}

Let
\begin{equation*}
 V_0(r)=\bigg(\Theta(w_*(r)^{-1})\bigg)^{u_0}r^{1-u_0/u}, \quad
 V_1(r)=\bigg(w_*(r)\Theta(w_*(r)^{-1})\bigg)^{u_1}r^{1-u_1/u}.
\end{equation*}

\noindent
\underline{Part 1}.
Proof of
\begin{equation*}
 \left\|U_0(r)\int_0^r F_0(t)\,dt\right\|_{L^{u/u_0}(0,\infty)}^{1/u_0}
 \le C
 \|f\|_{\Bwu}.
\end{equation*}
Case 1: $u_0<u<\infty$.
\begin{align*}
 &
 \left(\int_r^{\infty}U_0(t)^{u/u_0}\,dt\right)^{u_0/u}
 \left(\int_0^rV_0(t)^{-u/(u-u_0)}\,dt\right)^{(u-u_0)/u}
\\
 &=
 \left(\int_r^{\infty}\bigg(\Theta(w_*(t)^{-1})\bigg)^u\,\frac{dt}{t}\right)^{u_0/u}
 \left(\int_0^r\bigg(\Theta(w_*(t)^{-1})\bigg)^{-u_0u/(u-u_0)}\,\frac{dt}{t}\right)^{(u-u_0)/u}
\\
 &\ls
 \bigg(\Theta(w_*(t)^{-1})\bigg)^{u_0}
 \bigg(\Theta(w_*(t)^{-1})\bigg)^{-u_0}
 =1.
\end{align*}
Case 2: $u_0=u<\infty$.
\begin{align*}
 &
 \left(\int_r^{\infty}U_0(t)\,dt\right)
 \left(\sup_{t\in(0,r)}V_0(t)^{-1}\right)
\\
 &=
 \left(\int_r^{\infty}\bigg(\Theta(w_*(t)^{-1})\bigg)^{u_0}\,\frac{dt}{t}\right)
 \left(\sup_{t\in(0,r)}\bigg(\Theta(w_*(t)^{-1})\bigg)^{-u_0}\right)
\\
 &\ls
 \bigg(\Theta(w_*(r)^{-1})\bigg)^{u_0}\bigg(\Theta(w_*(t)^{-1})\bigg)^{-u_0}
 =1.
\end{align*}
Case 3: $u_0<u=\infty$.
In this case
\begin{equation*}
 U_0(r)=\bigg(\Theta(w_*(r)^{-1})\bigg)^{u_0}, \quad
 V_0(r)=\bigg(\Theta(w_*(r)^{-1})\bigg)^{u_0}r.
\end{equation*}
Then
\begin{align*}
 &
 \left(\sup_{t\in(r,\infty)}U_0(t)\right)
 \left(\int_0^rV_0(t)^{-1}\,dt\right)
\\
 &=
 \left(\sup_{t\in(r,\infty)}\bigg(\Theta(w_*(t)^{-1})\bigg)^{u_0}\right)
 \left(\int_0^r\bigg(\Theta(w_*(t)^{-1})\bigg)^{-u_0}\,\frac{dt}{t}\right)
\\
 &\sim
 \bigg(\Theta(w_*(r)^{-1})\bigg)^{u_0}\bigg(\Theta(w_*(r)^{-1})\bigg)^{-u_0}
 =1.
\end{align*}
Since
\begin{align*}
 V_0(r)F_0(r)
 &=
 \bigg(\Theta(w_*(r)^{-1})\bigg)^{u_0}r^{1-u_0/u}
 w_0(r)^{u_0}\|f\|_{E(Q_r)}^{u_0}r^{-1}
\\
 &=
 w(r)^{u_0}\|f\|_{E(Q_r)}^{u_0}r^{-u_0/u},
\end{align*}
using Theorem~\ref{thm:Muck}, we have
\begin{align*}
 \left\|U_0(r)\int_0^r F_0(t)\,dt\right\|_{L^{u/u_0}(0,\infty)}^{1/u_0}
 &\ls
 \left\|V_0(r)F_0(r)\right\|_{L^{u/u_0}(0,\infty)}^{1/u_0}
\\
 &=
 \left\|w(r)^{u_0}\|f\|_{E(Q_r)}^{u_0}r^{-u_0/u}
  \right\|_{L^{u/u_0}(0,\infty)}^{1/u_0}
\\
 &=
 \left\|w(r)\|f\|_{E(Q_r)}
  \right\|_{L^{u}((0,\infty),dr/r)}
 = 
 \|f\|_{\dB_{w}^{u}(E)}.
\end{align*}

\noindent
\underline{Part 2}.
Proof of
\begin{equation*}
 \left\|U_1(r)\int_r^{\infty}F_1(t)\,dt\right\|_{L^{u/u_1}(0,\infty)}^{1/u_1}
 \le C
 \|f\|_{\Bwu}.
\end{equation*}
Case 1: $u_1<u<\infty$.
\begin{align*}
 &
 \left(\int_0^rU_1(t)^{u/u_1}\,dt\right)^{u_1/u}
 \left(\int_r^{\infty}V_1(t)^{-u/(u-u_1)}\,dt\right)^{(u-u_1)/u}
\\
 &=
 \left(\int_0^r\bigg(w_*(t)\Theta(w_*(t)^{-1})\bigg)^{u}\,\frac{dt}{t}\right)^{u_1/u}
 \left(\int_r^{\infty}\bigg(w_*(t)\Theta(w_*(t)^{-1})\bigg)^{-u_1u/(u-u_1)}\,\frac{dt}{t}\right)^{(u-u_1)/u}
\\
 &\ls
 \bigg(w_*(r)\Theta(w_*(r)^{-1})\bigg)^{u_1}
 \bigg(w_*(r)\Theta(w_*(r)^{-1})\bigg)^{-u_1}
 =1.
\end{align*}
Case 2: $u_1=u<\infty$.
\begin{align*}
 &
 \left(\int_0^rU_1(t)\,dt\right)
 \left(\sup_{t\in(r,\infty)}V_1(t)^{-1}\right)
\\
 &=
 \left(\int_0^r\bigg(w_*(t)\Theta(w_*(t)^{-1})\bigg)^{u_1}\,\frac{dt}{t}\right)
 \left(\sup_{t\in(r,\infty)}\bigg(w_*(t)\Theta(w_*(t)^{-1})\bigg)^{-u_1}\right)
\\
 &\ls
 \bigg(w_*(r)\Theta(w_*(r)^{-1})\bigg)^{u_1}
 \bigg(w_*(r)\Theta(w_*(r)^{-1})\bigg)^{-u_1}
 =1.
\end{align*}
Case 3: $u_1<u=\infty$.
In this case
\begin{equation*}
 U_1(r)=\bigg(w_*(r)\Theta(w_*(r)^{-1})\bigg)^{u_1}, \quad
 V_1(r)=\bigg(w_*(r)\Theta(w_*(r)^{-1})\bigg)^{u_1}r.
\end{equation*}
Then
\begin{align*}
 &
 \left(\sup_{t\in(0,r)}U_1(t)\right)
 \left(\int_r^{\infty}V_1(t)^{-1}\,dt\right)
\\
 &=
 \left(\sup_{t\in(0,r)}\bigg(w_*(t)\Theta(w_*(t)^{-1})\bigg)^{u_1}\right)
 \left(\int_r^{\infty}\bigg(w_*(t)\Theta(w_*(t)^{-1})\bigg)^{-u_1}\,\frac{dt}{t}\right)
\\
 &\ls
 \bigg(w_*(r)\Theta(w_*(r)^{-1})\bigg)^{u_1}\bigg(w_*(r)\Theta(w_*(r)^{-1})\bigg)^{-u_1}
 =1.
\end{align*}
Since
\begin{align*}
 V_1(r)F_1(r)
 &=
 \bigg(w_*(r)\Theta(w_*(r)^{-1})\bigg)^{u_1}r^{1-u_1/u}
 w_1(r)^{u_1}\|f\|_{E(Q_r)}^{u_1}r^{-1}
\\
 &=
 w(r)^{u_1}\|f\|_{E(Q_r)}^{u_1}r^{-u_1/u},
\end{align*}
using Theorem~\ref{thm:Muck}, we have
\begin{align*}
 \left\|U_1(r)\int_r^{\infty}F_1(t)\,dt\right\|_{L^{u/u_1}(0,\infty)}^{1/u_1}
 &\ls
 \left\|V_1(r)F_1(r)\right\|_{L^{u/u_1}(0,\infty)}^{1/u_1}
\\
 &=
 \left\|w(r)^{u_1}
  \|f\|_{E(Q_t)}^{u_1}r^{-u_1/u}
  \right\|_{L^{u/u_1}(0,\infty)}^{1/u_1}
\\
 &=
 \left\|w(r)\|f\|_{E(Q_r)}
  \right\|_{L^{u}((0,\infty),dr/r)}
 =
 \|f\|_{\dB_{w}^{u}(E)}.
\end{align*}

(ii)
Since $U_0(r)=\Theta(w_*(r)^{-1})$ is almost decreasing,
\begin{align*}
 &\bigg\|U_0(r)
  \left(\sup_{t\in(0,r)}F_0(t)\right)
 \bigg\|_{L^{\infty}(0,\infty)}
 \ls
 \bigg\|
  \left(\sup_{t\in(0,r)}U_0(t)F_0(t)\right)
 \bigg\|_{L^{\infty}(0,\infty)}
\\
 &=
 \bigg\|
  U_0(t)F_0(t)
 \bigg\|_{L^{\infty}(0,\infty)}
 =
 \bigg\|w(t)\|f\|_{E(Q_t)}
 \bigg\|_{L^{\infty}(0,\infty)}
 = 
 \|f\|_{\dB_{w}^{\infty}(E)}.
\end{align*}

(iii)
Since $U_1(r)=w_*(r)\Theta(w_*(r)^{-1})$ is almost increasing,
\begin{align*}
 &\bigg\|U_1(r)
  \left(\sup_{t\in(r,\infty)}F_1(t)\right)
 \bigg\|_{L^{\infty}(0,\infty)}
 \ls
 \bigg\|
  \left(\sup_{t\in(r,\infty)}U_1(t)F_1(t)\right)
 \bigg\|_{L^{\infty}(0,\infty)}
\\
 &=
 \bigg\|
  U_1(t)F_1(t)
 \bigg\|_{L^{\infty}(0,\infty)}
 =
 \bigg\|w(t)\|f\|_{E(Q_t)}
 \bigg\|_{L^{\infty}(0,\infty)}
 = 
 \|f\|_{\dB_{w}^{\infty}(E)}.
\end{align*}
Therefore, we have the conclusion.
\end{proof}

\begin{proof}[Proof of Theorem~\ref{thm:IP}]

We may assume that 
$(w_0(r)/w_1(r))r^{-\epsilon}$ is almost increasing,
by changing $w_0$ and $w_1$ if need.

\noindent
\underline{Part 1}.
Proof of
\begin{equation}\label{subset}
 (\dB_{w_0}^{u_0}(E)(\R^n),\dB_{w_1}^{u_1}(E)(\R^n),\Theta)_{u}
 \subset\dB_{w}^{u}(E)(\R^n).
\end{equation}
Let $f\in (\dB_{w_0}^{u_0}(E)(\R^n),\dB_{w_1}^{u_1}(E)(\R^n),\Theta)_{u}$
and $f=f_0+f_1$ with $f_i\in\dB_{w_i}^{u_i}(E)(\R^n)$, $i=1,2$.
Then
\begin{align*}
 w(r)\|f\|_{E(Q_r)}
 &\le C
 w(r)\left(\|f_0\|_{E(Q_r)}+\|f_1\|_{E(Q_r)}\right)
\\
 &\le C
 \frac{w(r)}{w_0(r)}
 \left(
  w_0(r)\|f_0\|_{E(Q_r)}+\frac{w_0(r)}{w_1(r)}w_1(r)\|f_1\|_{E(Q_r)}
 \right)
\\
 &\le C
 \frac{w(r)}{w_0(r)}
 \left(
  \|f_0\|_{\dB_{w_0}^{\infty}(E)}
   +\frac{w_0(r)}{w_1(r)}\|f_1\|_{\dB_{w_1}^{\infty}(E)}
 \right)
\\
 &\le C\,
 \Theta\!\left(\frac{w_1(r)}{w_0(r)}\right)
 \left(
  \|f_0\|_{\dB_{w_0}^{u_0}(E)}
   +\frac{w_0(r)}{w_1(r)}\|f_1\|_{\dB_{w_1}^{u_1}(E)}
 \right).
\end{align*}
Then, letting $w_*=w_0/w_1$, we have
\begin{equation*}
 w(r)\|f\|_{E(Q_r)}
 \le C
 \Theta(w_*(r)^{-1})
 K(w_*(r),f;\dB_{w_0}^{u_0}(E)(\R^n),\dB_{w_1}^{u_1}(E)(\R^n)).
\end{equation*}
By Lemma~\ref{lem:phi} we have
\begin{align*}
 \|f\|_{\dB_{w}^{u}(E)}
 &=
 \|w(r)\|f\|_{E(Q_r)}\|_{L^u((0,\infty),dr/r)}
\\
 &\ls
 \|\Theta(w_*(r)^{-1})K(w_*(r),f;\dB_{w_0}^{u_0}(E)(\R^n),\dB_{w_1}^{u_1}(E)(\R^n))\|_{L^u((0,\infty),dr/r)}
\\
 &\sim
 \|\Theta(r^{-1})K(r,f;\dB_{w_0}^{u_0}(E)(\R^n),\dB_{w_1}^{u_1}(E)(\R^n))\|_{L^u((0,\infty),dr/r)}
\\
 &=
 \|f\|_{(\dB_{w_0}^{u_0}(E)(\R^n),\dB_{w_1}^{u_1}(E)(\R^n),\Theta)_{u}}.
\end{align*}
This shows \eqref{subset}.

\noindent
\underline{Part 2}.
Proof of
\begin{equation}\label{supset}
 (\dB_{w_0}^{u_0}(E)(\R^n),\dB_{w_1}^{u_1}(E)(\R^n),\Theta)_{u}
 \supset \dB_{w}^{u}(E)(\R^n).
\end{equation}
We may assume that $0<\max(u_0,u_1)\le u\le\infty$,
since
\begin{equation*}
 (\dB_{w_0}^{u_0}(E)(\R^n),\dB_{w_1}^{u_1}(E)(\R^n),\Theta)_{u}
 \supset
 (\dB_{w_0}^{\min(u_0,u)}(E)(\R^n),\dB_{w_1}^{\min(u_1,u)}(E)(\R^n),\Theta)_{u}.
\end{equation*}
Let $f\in\dB_{w}^{u}(E)(\R^n)$ and $r>0$. 
From the decomposition property of $\EQ$,
we can take functions $f_0^r$ and $f_1^r$ satisfying
$f=f_0^r+f_1^r$,
\begin{equation}\label{f0}
 \|f_0^r\|_{E(Q_t)}\le
\begin{cases}
 C_E\|f\|_{E(Q_t)} & (0<t<r), \\
 C_E\|f\|_{E(Q_{ar})} & (r\le t<\infty),
\end{cases}
\end{equation}
and
\begin{equation}\label{f1}
 \|f_1^r\|_{E(Q_t)}\le
\begin{cases}
 0 & (0<t<cr), \\
 C_E\|f\|_{E(Q_{bt})} & (cr\le t<\infty).
\end{cases}
\end{equation}
Here we may assume that $a\ge1$ and $b\ge1$.
We will show that 
$f_0^r\in \dB_{w_0}^{u_0}(E)(\R^n)$,
$f_1^r\in \dB_{w_1}^{u_1}(E)(\R^n)$ and
\begin{multline}\label{K<B}
 \left\|\Theta(w_*(r)^{-1})\|f_0^r\|_{\dB_{w_0}^{u_0}(E)}\right\|_{L^u((0,\infty),dr/r)}
  +\left\|w_*(r)\Theta(w_*(r)^{-1})\|f_1^r\|_{\dB_{w_1}^{u_1}(E)}\right\|_{L^u((0,\infty),dr/r)}
\\
 \ls
 \|f\|_{\dB_{w}^{u}(E)}.
\end{multline}
Then, by Lemma~\ref{lem:phi}
\begin{align*}
 &\|f\|_{(\dB_{w_0}^{u_0}(E)(\R^n),\dB_{w_1}^{u_1}(E)(\R^n),\Theta)_{u}}
\\
 &=
 \left\|\Theta(r^{-1})
  K(r,f;\dB_{w_0}^{u_0}(E)(\R^n),\dB_{w_1}^{u_1}(E)(\R^n))
 \right\|_{L^u((0,\infty),dr/r)}
\\
 &\sim
 \left\|\Theta(w_*(r)^{-1})
  K(w_*(r),f;\dB_{w_0}^{u_0}(E)(\R^n),\dB_{w_1}^{u_1}(E)(\R^n))
 \right\|_{L^u((0,\infty),dr/r)}
\\
 &\le
 \left\|\Theta(w_*(r)^{-1})
  \|f_0^r\|_{\dB_{w_0}^{u_0}(E)}
  +w_*(r)\Theta(w_*(r)^{-1})\|f_1^r\|_{\dB_{w_1}^{u_1}(E)}
 \right\|_{L^u((0,\infty),dr/r)}
\\
 &\ls
 \|f\|_{\dB_{w}^{u}(E)}.
\end{align*}
This shows \eqref{supset}.

Now we prove \eqref{K<B}.
From Lemma~\ref{lem:w*} we see that 
\begin{equation*}
 \|w_0(t)\|f\|_{E(Q_t)}\|_{L^{u_0}((0,2ar),dt/t)}<\infty, \quad
 \|w_1(t)\|f\|_{E(Q_t)}\|_{L^{u_1}([r,\infty),dt/t)}<\infty,
\end{equation*}
and
\begin{multline*}
 \left\|\Theta(w_*(2ar)^{-1})
  \|w_0(t)\|f\|_{E(Q_t)}\|_{L^{u_0}((0,2ar),dt/t)}
 \right\|_{L^u((0,\infty),dr/r)}
\\
 +\left\|w_*(r)\Theta(w_*(r)^{-1})
  \|w_1(t)\|f\|_{E(Q_t)}\|_{L^{u_1}([r,\infty),dt/t)}
 \right\|_{L^u((0,\infty),dr/r)}
 \ls
 \|f\|_{\dB_{w}^{u}(E)}.
\end{multline*}
Therefore, to prove \eqref{K<B} it is enough to show
\begin{equation}\label{w0}
 \|w_0(t)\|f_0^r\|_{E(Q_t)}\|_{L^{u_0}((0,\infty),dt/t)}
 \ls
 \|w_0(t)\|f\|_{E(Q_t)}\|_{L^{u_0}((0,2ar),dt/t)},
\end{equation}
and
\begin{equation}\label{w1}
 \|w_1(t)\|f_1^r\|_{E(Q_t)}\|_{L^{u_1}((0,\infty),dt/t)}
 \ls
 \|w_1(t)\|f\|_{E(Q_t)}\|_{L^{u_1}([r,\infty),dt/t)}.
\end{equation}
Since $w_0\in\cW^{\infty}$ if $u_0=\infty$,
or $w_0\in\cW^*$ if $u_0<\infty$,
\begin{equation*}
 \|w_0(t)\|_{L^{u_0}([r,\infty),dt/t)}
 \ls
 w_0(r)
 \ls
 \|w_0(t)\|_{L^{u_0}([ar,2ar),dt/t)}.
\end{equation*}
From \eqref{f0} it follows that
\begin{align*}
 &
 \|w_0(t)\|f_0^r\|_{E(Q_t)}\|_{L^{u_0}((0,\infty),dt/t)}
\\
 &\ls
 \|w_0(t)\|f\|_{E(Q_t)}\|_{L^{u_0}((0,r),dt/t)}
 +\|f\|_{E(Q_{ar})}\|w_0(t)\|_{L^{u_0}([r,\infty),dt/t)}
\\
 &\ls
 \|w_0(t)\|f\|_{E(Q_t)}\|_{L^{u_0}((0,2ar),dt/t)}.
\end{align*}
This shows \eqref{w0}.
Next we show \eqref{w1}.
From \eqref{f1} it follows that
\begin{align*}
 \|w_1(t)\|f_1^r\|_{E(Q_t)}\|_{L^{u_1}((0,\infty),dt/t)}
 &\ls
 \|w_1(t)\|f\|_{E(Q_{bt})}\|_{L^{u_1}([cr,\infty),dt/t)}
\\
 &\sim
 \|w_1(bt)\|f\|_{E(Q_{bt})}\|_{L^{u_1}([cr,\infty),dt/t)}
\\
 &=
 \|w_1(t)\|f\|_{E(Q_{t})}\|_{L^{u_1}([cr/b,\infty),dt/t)}.
\end{align*}
If $c/b\ge1$, then we have \eqref{w1}. 
If $c/b<1$, then 
\begin{align*}
 &
 \|w_1(t)\|f\|_{E(Q_{t})}\|_{L^{u_1}([cr/b,\infty),dt/t)}
\\
 &=
 \|w_1(t)\|f\|_{E(Q_{t})}\|_{L^{u_1}([cr/b,r),dt/t)}
 +\|w_1(t)\|f\|_{E(Q_{t})}\|_{L^{u_1}([r,\infty),dt/t)}
\\
 &\ls
 \|w_1(t)\|f\|_{E(Q_{t})}\|_{L^{u_1}([r,br/c),dt/t)}
 +\|w_1(t)\|f\|_{E(Q_{t})}\|_{L^{u_1}([r,\infty),dt/t)}
\\
 &\le2
 \|w_1(t)\|f\|_{E(Q_{t})}\|_{L^{u_1}([r,\infty),dt/t)}.
\end{align*}
This shows \eqref{w1}.

\noindent
\underline{Part 3}.
Proof of
\begin{equation}\label{subset B}
 (B_{w_0}^{u_0}(E)(\R^n),B_{w_1}^{u_1}(E)(\R^n),\Theta)_{u,\,[1,\infty)}
 \subset
 B_{w}^{u}(E)(\R^n).
\end{equation}
Using $L^u([1,\infty),dr/r)$ instead of $L^u((0,\infty),dr/r)$ in Part 1,
we have the conclusion.

\noindent
\underline{Part 4}.
Proof of
\begin{equation}\label{supset B}
 (B_{w_0}^{u_0}(E)(\R^n),B_{w_1}^{u_1}(E)(\R^n),\Theta)_{u,\,[1,\infty)}
 \supset
 B_{w}^{u}(E)(\R^n).
\end{equation}
Instead of \eqref{K<B} we need 
\begin{multline}\label{K<B 1}
 \left\|
  \Theta(w_*(r)^{-1})\|f_0^r\|_{B_{w_0}^{u_0}(E)}
 \right\|_{L^u([1,\infty),dr/r)}
 +\left\|
    w_*(r)\Theta(w_*(r)^{-1})\|f_1^r\|_{B_{w_1}^{u_1}(E)}
  \right\|_{L^u([1,\infty),dr/r)}
\\
 \ls
 \|f\|_{B_{w}^{u}(E)}.
\end{multline}
By the same way as 
\eqref{w0} and \eqref{w1}
we can get, for $r\ge1$, 
\begin{equation*}\label{w0 1}
 \|w_0(t)\|f_0^r\|_{E(Q_t)}\|_{L^{u_0}([1,\infty),dt/t)}
 \ls
 \|w_0(t)\|f\|_{E(Q_t)}\|_{L^{u_0}([1,2ar),dt/t)},
\end{equation*}
and
\begin{equation*}\label{w1 1}
 \|w_1(t)\|f_1^r\|_{E(Q_t)}\|_{L^{u_1}([1,\infty),dt/t)}
 \ls
 \|w_1(t)\|f\|_{E(Q_t)}\|_{L^{u_1}([r,\infty),dt/t)},
\end{equation*}
respectively.
By Remark~\ref{rem:w*} 
we see that \eqref{K<B 1} follows from these inequalities. 
\end{proof}

\section{Boundedness of linear and sublinear operators}\label{s:bdd}

In this section we consider the boundedness of
linear and sublinear operators
on $\Bwu(\R^n)$ and $\dBwu(\R^n)$ 
with $E=L_{p,\lambda}$ or $\cL_{p,\lambda}$.
It is known that some classical operators are bounded on 
$B_{\sigma}(E)(\R^n)$ and $\dB_{\sigma}(E)(\R^n)$,
see \cite{KoMaNaSa2013RMC}. 
Applying the interpolation property, 
we extend these boundedness to $\Bwu(\R^n)$ and $\dBwu(\R^n)$.
We consider sublinear operators $T$ defined on $L^1_{\comp}(\R^n)$.
That is, the operator $T$ satisfies that, 
for all $f,g\in L^1_{\comp}(\R^n)$ and for a.e.\,$x\in\R^n$,
\begin{equation*}\label{subadd}
 |T(f+g)(x)|\le|Tf(x)|+|Tg(x)|.
\end{equation*}
We also assume that 
\begin{equation}\label{subdiff}
 |Tf(x)-Tg(x)|\le C|T(f-g)(x)|
\end{equation}
for some positive constant $C$.
For example, if $T$ is linear, or, 
sublinear and $Tf(x)\ge0$ for all $f$ and a.e.\,$x$,
then $T$ satisfies the condition \eqref{subdiff} with $C=1$. 

In general, for quasi-normed function spaces $A_i$ and $B_i$, $i=0,1$,
let a sublinear operator $T:A_0+A_1\to B_0+B_1$
be bounded from $A_i$ to $B_i$, $i=0,1$,
and satisfy \eqref{subdiff}
for all $f,g\in A_0+A_1$.
If $T$ is not linear,
we also assume that $B_i$, $i=0,1$, satisfy the lattice property \eqref{lattice}.
Then we conclude that
\begin{equation*}
 K(r,Tf;B_0,B_1) \le C_TK(r,f;A_0,A_1),
\end{equation*}
where $C_T$ is a positive constant dependent on $T$ and $C$ in \eqref{subdiff}.
Therefore we can use the interpolation property for the boundedness of $T$.
Actually, if $T$ is linear, then
\begin{equation*}
 Tf=Tf_0+Tf_1, \
 Tf_0\in B_0, \ Tf_1\in B_1
\end{equation*}
for any decomposition $f=f_0+f_1$ in $A_0+A_1$.
Hence
\begin{equation*}
 K(r,Tf;B_0,B_1)
 \le
 \|Tf_0\|_{B_0}+r\|Tf_1\|_{B_1}
 \le
 C_T(\|f_0\|_{A_0}+r\|f_1\|_{A_1}).
\end{equation*}
If $T$ is not linear,
then, using \eqref{subdiff} and the lattice property,
we have
\begin{equation*}
 |Tf(x)-Tf_0(x)|\le C|Tf_1(x)|
\end{equation*}
and 
\begin{equation*}
 Tf=Tf_0+(Tf-Tf_0), \
 Tf_0\in B_0, \ Tf-Tf_0\in B_1.
\end{equation*}
Hence
\begin{equation*}
 K(r,Tf;B_0,B_1)
 \le
 \|Tf_0\|_{B_0}+r\|Tf-Tf_0\|_{B_1}
 \le
 C_T(\|f_0\|_{A_0}+r\|f_1\|_{A_1}),
\end{equation*}
for any decomposition $f=f_0+f_1$ in $A_0+A_1$.

We also point out that the condition \eqref{subdiff} is important
to extend $L^p$-bounded operators to bounded operators on Morrey spaces.
Actually, there exists an $L^p$-bounded sublinear operator $T$ such that
$T$ does not satisfy \eqref{subdiff} and 
that $T$ cannot be extended to a bounded operator on Morrey spaces,
see Remark~\ref{rem:def Mor}.

In this section,
first we give the boundedness of 
the Hardy-Littlewood maximal and fractional maximal operators
in Subsection~\ref{ss:M}.
Next we investigate singular and fractional integral operators
and more general sublinear operators with \eqref{subdiff} 
in Subsection~\ref{ss:SI}.
In Subsections~\ref{ss:SICP} and \ref{ss:tFI}
we consider singular integral operators with the cancellation property
and modified fractional integral operators, respectively.
Finally, we show the vector-valued boundedness 
in Subsection~\ref{ss:Vector}.

If $\lambda=-n/p$,
then $L_{p,\lambda}=L^p$ 
and $\dB_w^u(L_{p,\lambda})(\R^n)=\dB_w^u(L^p)(\R^n)=LM_{pu,\tw}(\R^n)$ 
with $\tw(r)=w(r)/r$.
Let $\dB_w^u(WL^p)(\R^n)=WLM_{pu,\tw}(\R^n)$ with $\tw(r)=w(r)/r$,
where $WL^p$ is the weak $L^p$ space.

\subsection{The Hardy-Littlewood maximal and fractional maximal operators}\label{ss:M}

The fractional maximal operators $M_{\alpha}$ 
of order $\alpha\in[0,n)$ are sublinear, 
which is defined as
\begin{equation*}\label{frM}
 M_{\alpha}f(x)=\sup_{Q\ni x}\frac1{|Q|^{1-\alpha/n}}\int_Q|f(y)|\,dy,
\end{equation*}
where the supremum is taken over all cubes (or balls) $Q$ 
containing $x\in\R^n$.
If $\alpha=0$, 
then $M_{\alpha}$ is the Hardy-Littlewood maximal operator denoted by $M$.

It is known that, 
for $\alpha\in[0,n)$, $p,q\in[1,\infty]$ and $-n/p+\alpha=-n/q$,  
the operator $M_{\alpha}$ is bounded 
from $L^p(\R^n)$ to $L^q(\R^n)$ if $p\in(1,\infty]$,
and from $L^1(\R^n)$ to $WL^q(\R^n)$ if $p=1$.

It is also known that, 
for $\alpha\in[0,n)$, $p,q\in[1,\infty)$,
$\lambda\in[-n/p,0)$, $\mu\in[-n/q,0)$,
$\mu=\lambda+\alpha$ and $q\le(\lambda/\mu)p$,
the operator $M_{\alpha}$ is bounded 
from $L_{p,\lambda}(\R^n)$ to $L_{q,\mu}(\R^n)$ if $p\in(1,\infty)$,
and from $L_{1,\lambda}(\R^n)$ to $WL_{q,\mu}(\R^n)$ if $p=1$.
In particular,
the Hardy-Littlewood maximal operator $M$ is bounded 
from  $L_{p,\lambda}(\R^n)$ to itself if $p\in(1,\infty)$ and
from $L_{1,\lambda}(\R^n)$ to $WL_{1,\lambda}(\R^n)$,
see \cite{ChiarenzaFrasca1987}.

The following is known:

\begin{thm}[\cite{KoMaNaSa2013RMC}]\label{thm:M:Bs(Mor)}
Let $\alpha\in[0,n)$, $\sigma\in[0,\infty)$ and
$p,q\in[1,\infty)$,
and let $\lambda\in[-n/p,0)$ and $\mu\in[-n/q,0)$.
Assume that
\begin{equation*}
 \mu=\lambda+\alpha, \quad q\le(\lambda/\mu)p  
 \quad\text{and}\quad \sigma+\lambda+\alpha\le 0.
\end{equation*}
Then the operator $M_{\alpha}$ is bounded 
from $B_{\sigma}(L_{p,\lambda})(\R^n)$ to $B_{\sigma}(L_{q,\mu})(\R^n)$
if $p\in(1,\infty)$, 
from $B_{\sigma}(L_{1,\lambda})(\R^n)$ to $B_{\sigma}(WL_{q,\mu})(\R^n)$
if $p=1$.
The same conclusion holds for $\dB_{\sigma}(L_{p,\lambda})(\R^n)$.
\end{thm}

\begin{rem}
Let $\alpha=0$ in the theorem above.
Then we get the boundedness of the Hardy-Littlewood maximal operator $M$ 
on $B_{\sigma}(L_{p,\lambda})(\R^n)$ if $p\in(1,\infty)$,
and 
from $B_{\sigma}(L_{1,\lambda})(\R^n)$ to $B_{\sigma}(WL_{1,\lambda})(\R^n)$
if $p=1$.
\end{rem}

Using Theorem~\ref{thm:M:Bs(Mor)} and Example~\ref{exmp:MC3},
we have the following:

\begin{thm}\label{thm:M:Bwu(Mor)}
Let $\alpha\in[0,n)$, 
$p,q\in[1,\infty)$,
$\lambda\in[-n/p,0)$, $\mu\in[-n/q,0)$,
$u\in(0,\infty]$, $\Theta\in\iTheta$,
and let
\begin{equation*}\label{M w}
 w(r)=r^{-\sigma}\Theta(r^{\tau}),  \quad
 \sigma,\tau\in(0,\infty) \ \text{with} \ \sigma>\tau.
\end{equation*}
Assume that
\begin{equation*}
 \mu=\lambda+\alpha, \quad q\le(\lambda/\mu)p  
 \quad\text{and}\quad \sigma+\lambda+\alpha\le 0.
\end{equation*}
Then the operator $M_{\alpha}$ is bounded 
from $B_w^u(L_{p,\lambda})(\R^n)$ to $B_w^u(L_{q,\mu})(\R^n)$
if $p\in(1,\infty)$, 
from $B_w^u(L_{1,\lambda})(\R^n)$ to $B_w^u(WL_{q,\mu})(\R^n)$
if $p=1$.
The same conclusion holds for $\dB_w^u(L_{p,\lambda})(\R^n)$.
\end{thm}

Taking $\lambda=-n/p$ and $\mu=-n/q$ in Theorem~\ref{thm:M:Bwu(Mor)},
we have the following:

\begin{cor}\label{cor:M:LocM}
Let $\alpha\in[0,n)$,
$p,q\in[1,\infty)$, $u\in(0,\infty]$, $\Theta\in\iTheta$,
and let
\begin{equation*}
 \tw(r)=w(r)/r, \quad
 w(r)=r^{-\sigma}\Theta(r^{\tau}),  \quad
 \sigma,\tau\in(0,\infty) \ \text{with} \ \sigma>\tau.
\end{equation*}
Assume that
\begin{equation*}
 -n/q=-n/p+\alpha  
 \quad\text{and}\quad \sigma-n/p+\alpha\le 0.
\end{equation*}
Then the operator $M_{\alpha}$ is bounded 
from $LM_{pu,\tw}(\R^n)$ to $LM_{qu,\tw}(\R^n)$
if $p\in(1,\infty)$, 
from $LM_{1u,\tw}(\R^n)$ to $WLM_{qu,\tw}(\R^n)$
if $p=1$.
\end{cor}

For necessary and sufficient conditions for the boundedness of $M$
on local Morrey-type spaces, see \cite{BuGu2004}.

\subsection{Singular and fractional integral operators}\label{ss:SI}

We consider sublinear operators $T$ 
which satisfy \eqref{subdiff} and
the following condition:
There exist constants $\alpha\in[0,n)$ and $C\in(0,\infty)$
such that, for all $f\in L^1_{\comp}(\R^n)$,
\begin{equation}\label{size}
 |Tf(x)|
 \le 
 C\int_{\R^n} \frac{|\Omega(x-y)|}{|x-y|^{n-\alpha}}|f(y)|\,dy,
 \quad
 x\notin \supp f,
\end{equation}
where $\Omega$ is a function on $\R^n$ which is homogeneous of degree zero 
and $\Omega\in L^{\tp}(S^{n-1})$ for some $\tp\in[1,\infty]$.
For example, 
singular and fractional integral operators satisfy \eqref{size} with $\Omega\equiv1$. 
More precisely,
the singular integral operator $T$ is defined by 
\begin{equation}\label{SI}
 T f(x)=\int_{\R^n} K(x,y)f(y)\,dy,
 \quad
 x\notin \supp f,
 \quad
 f\in L^1_{\comp}(\R^n)
\end{equation}
with kernel $K(x,y)$ satisfying the condition
\begin{equation}\label{K}
 |K(x,y)|\le C|x-y|^{-n}, \quad x\ne y, 
\end{equation}
and some regularity conditions.
(For regularity conditions,
see Yabuta \cite{Yabuta1985} and references therein.)
Then the singular integral operator $T$ satisfies the condition \eqref{size} with $\alpha=0$
and
it is bounded on $L^p(\R^n)$, $p\in(1,\infty)$,
and from $L^1(\R^n)$ to $WL^1(\R^n)$.
Moreover, 
under the assumption that
$p\in[1,\infty)$ and $\lambda\in[-n/p,0)$,
$T$ can be extended to a bounded operator on $L_{p,\lambda}(\R^n)$ if $p\in(1,\infty)$,
and 
from $L_{1,\lambda}(\R^n)$ to $WL_{1,\lambda}(\R^n)$ if $p=1$, 
see \cite{ChiarenzaFrasca1987,Nakai1994MN,Peetre1966}.
Fractional integral operators $I_{\alpha}$, $\alpha\in(0,n)$, are defined by
\begin{equation*}\label{FI}
 I_{\alpha}f(x)=\int_{\R^n} \frac{f(y)}{|x-y|^{n-\alpha}}\,dy.
\end{equation*}
Then $I_{\alpha}$ satisfies \eqref{size} with this $\alpha$
and 
it is bounded from $L^p(\R^n)$ to $L^q(\R^n)$, $1<p<q<\infty$, $-n/p+\alpha=-n/q$,
and from $L^1(\R)$ to $WL^{n/(n-\alpha)}(\R)$.
Moreover, 
under the assumption that
$p,q\in[1,\infty)$, $\lambda\in[-n/p,0)$, $\mu\in[-n/q,0)$,
$\lambda+\alpha=\mu$ and $q\le(\lambda/\mu)p$,
$I_{\alpha}$ can be extended to a bounded operator 
from $L_{p,\lambda}(\R^n)$ to $L_{q,\mu}(\R^n)$ if $p\in(1,\infty)$,  
and from $L_{1,\lambda}(\R^n)$ to $WL_{q,\mu}(\R^n)$ if $p=1$,
see \cite{Adams1975,ChiarenzaFrasca1987}.

For the $L^p$-boundedness of Calder\'on-Zygmund singular integral operators
\begin{equation*}\label{CZ Omega}
 T_{\Omega}f(x)=p.v.\int_{\R^n} \frac{\Omega(x-y)}{|x-y|^{n}}f(y)\,dy,
\end{equation*}
and
fractional integral operators with rough kernel
\begin{equation*}\label{FI Omega}
 I_{\Omega,\alpha}f(x)=\int_{\R^n} \frac{\Omega(x-y)}{|x-y|^{n-\alpha}}f(y)\,dy,
\end{equation*}
see \cite{CalderonZygmund1956} and \cite{Muckenhoupt1960}, respectively.
See also 
\cite{DingYangZhow1998,GuliyevAliyevKaramanShukurov2011,LuLuYang2002,SorWei1994},
for C.~Fefferman's singular multipliers, 
Ricci-Stein's oscillatory singular integral, 
the Littlewood-Paley operator, Marcinkiewicz operator,
the Bochner-Riesz operator at the critical index and so on.

\begin{rem}\label{rem:def Mor}
Let $T$ be a sublinear operator satisfying 
\eqref{subdiff} and \eqref{size} for some $\alpha\in[0,n)$. 
Let $p,q\in[1,\infty)$, $\lambda\in[-n/p,0)$, $\mu\in[-n/q,0)$ 
and $\mu=\lambda+\alpha$.
Assume that $T$ is bounded 
from $L^p(\R^n)$ to $L^q(\R^n)$ or to $WL^q(\R^n)$.
Then, for $f\in L_{p,\lambda}(\R^n)$ and $R>0$,
$T(f\chi_R)$ is well defined and 
$\lim_{R\to\infty}T(f\chi_R)$ exists a.e.\,on $\R^n$, or in $L^q_{\loc}(\R^n)$,
with some additional assumption on $\Omega$ in \eqref{size}.
Actually, 
$f\chi_R\in L^p(\R^n)$ and 
we can prove that
\begin{equation*}
 |T(f\chi_S)(x)-T(f\chi_R)(x)|\le C|T(f(\chi_S-\chi_R)(x)|\to0
\end{equation*}
as $R,S\to\infty$ for a.e.\ $\R^n$, or in $L^q_{\loc}(\R^n)$,
see \cite[Lemmas~3 and 4]{KoMaNaSa2013RMC}.
Then, 
letting $Tf=\lim_{R\to\infty}T(f\chi_R)$ for $f\in L_{p,\lambda}(\R^n)$,
we can define $T$ as a bounded operator 
from $L_{p,\lambda}(\R^n)$ to $L_{q,\mu}(\R^n)$ or to $WL_{q,\mu}(\R^n)$,
see \cite[Remark~15]{KoMaNaSa2013RMC} 
in which we point out that we need the condition \eqref{subdiff}.
For example, 
the operator $Tf=e^{i\|f\|_{L^p(\R^n)}}Mf$, 
where $M$ is the Hardy-Littlewood maximal operator,
is bounded on $L^p(\R^n)$ but not well defined on Morrey spaces in general.
\end{rem}

\begin{rem}\label{rem:def T Mor}
If $T$ is a singular integral operator defined by \eqref{SI},
then the equality
\begin{equation*}
 \lim_{R\to\infty}T(f\chi_R)(x)
 =
 T(f\chi_{Q(z,2r)})(x)+\int_{\R^n\setminus Q(z,2r)}K(x,y)f(y)\,dy
\end{equation*}
holds for a.e.\,$x\in Q(z,r)$ and for any $Q(z,r)$,
see \cite{Nakai1994MN,Nakai2010RMC,SawanoTanaka2005}.
See also Rosenthal and Triebel~\cite{RosenthalTriebel2014} 
for the extension of singular integral (Calder\'on-Zygmund) operators 
to Morrey spaces.
\end{rem}

\begin{thm}[\cite{KoMaNaSa2013RMC}]\label{thm:T:Bs(Mor)}
Let $\sigma\in[0,\infty)$ and $p,q\in[1,\infty)$, 
and let $\lambda\in[-n/p,0)$ and $\mu\in[-n/q,0)$.
Let $T$ be a sublinear operator defined on $L^1_{\comp}(\R^n)$ and
satisfy \eqref{subdiff} and 
\eqref{size} for some $\alpha\in[0,n)$ 
and $\Omega\in L^{\tp}(S^{n-1})$
with $\tp\in[1,\infty]$. 
Assume one of the following conditions:
\begin{enumerate}
\item $\mu=\lambda+\alpha$, $\tp\ge p'$ and $\sigma+\lambda+\alpha<0$,
\item $\mu=\lambda+\alpha$, $\tp\ge q$ and $\sigma+\lambda+n/\tp+\alpha<0$.
\end{enumerate}
Assume in addition $T$ can be extended to a bounded operator 
from $L_{p,\lambda}(\R^n)$ to $L_{q,\mu}(\R^n)$ 
or to $WL_{q,\mu}(\R^n)$.
Then $T$ can be further extended to a bounded operator 
from $B_{\sigma}(L_{p,\lambda})(\R^n)$ to $B_{\sigma}(L_{q,\mu})(\R^n)$
or to $B_{\sigma}(WL_{q,\mu})(\R^n)$, respectively.
The same conclusion holds for $\dB_{\sigma}(L_{p,\lambda})(\R^n)$.
\end{thm}

\begin{rem}\label{rem:def Bs}
Under the assumption in Theorem~\ref{thm:T:Bs(Mor)},
$\lim_{R\to\infty}T(f\chi_R)$ exists a.e.\,on $\R^n$, or in $L^q_{\loc}(\R^n)$, 
for $f\in B_{\sigma}(L_{p,\lambda})(\R^n)$
or $f\in\dB_{\sigma}(L_{p,\lambda})(\R^n)$.
Then, letting $Tf=\lim_{R\to\infty}T(f\chi_R)$, we have the desired boundedness
(see \cite[Subsection~6.4]{KoMaNaSa2013RMC}).
\end{rem}

In Theorem~\ref{thm:T:Bs(Mor)} we cannot take $\sigma+\lambda+\alpha=0$ 
differently from Theorem~\ref{thm:M:Bs(Mor)},
see \cite[Remark~9]{KoMaNaSa2013RMC}.

Using Theorem~\ref{thm:T:Bs(Mor)} and Example~\ref{exmp:MC3},
we have the following:

\begin{thm}\label{thm:T:Bwu(Mor)}
Let $p,q\in[1,\infty)$, 
$\lambda\in[-n/p,0)$, $\mu\in[-n/q,0)$,
$u\in(0,\infty]$, $\Theta\in\iTheta$,
and let
\begin{equation}\label{T w}
 w(r)=r^{-\sigma}\Theta(r^{\tau}),  \quad
 \sigma,\tau\in(0,\infty) \ \text{with} \ \sigma>\tau.
\end{equation}
Let $T$ be a sublinear operator defined on $L^1_{\comp}(\R^n)$ and
satisfy \eqref{subdiff} and 
\eqref{size} for some $\alpha\in[0,n)$ 
and $\Omega\in L^{\tp}(S^{n-1})$
with $\tp\in[1,\infty]$. 
Assume one of the following conditions:
\begin{enumerate}
\item $\mu=\lambda+\alpha$, $\tp\ge p'$ and $\sigma+\lambda+\alpha<0$,
\item $\mu=\lambda+\alpha$, $\tp\ge q$ and $\sigma+\lambda+n/\tp+\alpha<0$.
\end{enumerate}
Assume in addition $T$ can be extended to a bounded operator 
from $L_{p,\lambda}(\R^n)$ to $L_{q,\mu}(\R^n)$ 
or to $WL_{q,\mu}(\R^n)$.
Then $T$ can be further extended to a bounded operator 
from $B_w^u(L_{p,\lambda})(\R^n)$ to $B_w^u(L_{q,\mu})(\R^n)$
or to $B_w^u(WL_{q,\mu})(\R^n)$, respectively.
The same conclusion holds for $\dB_w^u(L_{p,\lambda})(\R^n)$.
\end{thm}

\begin{rem}\label{rem:def Bwu}
Let $f\in B_w^u(L_{p,\lambda})(\R^n)$ and $R>0$.
Then $f\chi_R\in B_{\tau}(L_{p,\lambda})(\R^n)$ and $f(1-\chi_R)\in B_{\sigma}(L_{p,\lambda})(\R^n)$.
Hence
$T(f\chi_R)$ and $T(f(1-\chi_R))$ are well defined by Theorem~\ref{thm:T:Bs(Mor)}.
Moreover, by Remark~\ref{rem:def Bs}, 
$Tf=\lim_{R\to\infty}T(f\chi_R)$ is well defined. 
\end{rem}

\begin{cor}\label{cor:SI:Bwu(Mor)}
Let $T$ be a singular integral operator 
with kernel $K(x,y)$ satisfying the condition \eqref{K}.
Let $p\in[1,\infty)$, $\lambda\in[-n/p,0)$,
$u\in(0,\infty]$, $\Theta\in\iTheta$,
and define $w$ by \eqref{T w}.
Assume that $\sigma+\lambda<0$.
If $T$ is bounded on $L^{p}(\R^n)$ with $p\in(1,\infty)$, 
then $T$ can be extended to a bounded operator on $B_w^u(L_{p,\lambda})(\R^n)$.
If $T$ is bounded from $L^1(\R^n)$ to $WL^1(\R^n)$,
then $T$ can be extended to a bounded operator 
from $B_w^u(L_{1,\lambda})(\R^n)$ to $B_w^u(WL_{1,\lambda})(\R^n)$.
The same conclusion holds for $\dB_w^u(L_{p,\lambda})(\R^n)$.
\end{cor}

\begin{cor}\label{cor:FI:Bwu(Mor)}
Let $\alpha\in(0,n)$, $p,q\in[1,\infty)$, 
$\lambda\in[-n/p,0)$, $\mu\in[-n/q,0)$,
$u\in(0,\infty]$, $\Theta\in\iTheta$,
and define $w$ by \eqref{T w}.
Assume that $\lambda+\alpha=\mu$, 
$q\le(\lambda/\mu)p$ and $\sigma+\mu<0$.
Then fractional integral operators $I_{\alpha}$ are bounded 
from $B_w^u(L_{p,\lambda})(\R^n)$ to $B_w^u(L_{q,\mu})(\R^n)$ if $p\in(1,\infty)$, 
and from $B_w^u(L_{1,\lambda})(\R^n)$ to $B_w^u(WL_{q,\mu})(\R^n)$ if $p=1$.
The same conclusion holds for $\dB_w^u(L_{p,\lambda})(\R^n)$.
\end{cor}

Further,
Theorem~\ref{thm:T:Bwu(Mor)} is valid for 
the Calder\'on-Zygmund singular integral operators,
fractional integral operators with rough kernel,
C.~Fefferman's singular multipliers, 
the Littlewood-Paley operator, the Marcinkiewicz operator,
Ricci-Stein's oscillatory singular integral, 
the Bochner-Riesz operator at the critical index, and so on.

Taking $\lambda=-n/p$ and $\mu=-n/q$ in Theorem~\ref{thm:T:Bwu(Mor)},
we have the following:

\begin{cor}\label{cor:T:locM)}
Let $p,q\in[1,\infty)$, 
$u\in(0,\infty]$, $\Theta\in\iTheta$,
and let
\begin{equation*}\label{T tw}
 \tw(r)=w(r)/r, \quad
 w(r)=r^{-\sigma}\Theta(r^{\tau}),  \quad
 \sigma,\tau\in(0,\infty) \ \text{with} \ \sigma>\tau.
\end{equation*}
Let $T$ be a sublinear operator defined on $L^1_{\comp}(\R^n)$ and
satisfy \eqref{subdiff} and 
\eqref{size} for some $\alpha\in[0,n)$ 
and $\Omega\in L^{\tp}(S^{n-1})$
with $\tp\in[1,\infty]$. 
Assume one of the following conditions:
\begin{enumerate}
\item $-n/q=-n/p+\alpha$, $\tp\ge p'$ and $\sigma-n/p+\alpha<0$,
\item $-n/q=-n/p+\alpha$, $\tp\ge q$ and $\sigma-n/p+n/\tp+\alpha<0$.
\end{enumerate}
Assume in addition $T$ is a bounded operator 
from $L^p(\R^n)$ to $L^q(\R^n)$ 
or to $WL^q(\R^n)$.
Then $T$ can be extended to a bounded operator 
from $LM_{pu,\tw}(\R^n)$ to $LM_{qu,\tw}(\R^n)$
or to $WLM_{qu,\tw}(\R^n)$, respectively.
\end{cor}

For the boundedness of singular and fractional integral operators
on local Morrey-type spaces, see \cite{BuGu2009,BuGu2010}.

\subsection{Singular integral operators with the cancellation property}\label{ss:SICP}

Let $\kappa\in(0,1]$.
In this section we consider a singular integral operator $T$  with kernel $K(x,y)$
satisfying the following properties;
\begin{gather*} 
  |K(x,y)|\le \frac{C}{|x-y|^n}
     \quad\text{for}\ x\not=y;
                                                  \label{SK1}
 \\
  \begin{split}
  |K(x,y)-K(z,y)|+|K(y,x)-K(y,z)| 
  \le &
  \frac{C}{|x-y|^n}
  \left(\frac{|x-z|}{|x-y|}\right)^{\kappa} \\
  &\text{for}\ |x-y|\ge2|x-z|;
                                                  \label{SK2}
  \end{split}
 \\
  \begin{split}
  \int_{r\le|x-y|<R} K(x,y) \,dy
  =\int_{r\le|x-y|<R} & K(y,x) \,dy = 0  \\
  &\text{for}\ 0<r<R<\infty\ \text{and}\ x\in \R^n,
                                                  \label{SK3}
  \end{split}
\end{gather*}
where $C$ is a positive constant independent of $x,y,z\in\R^n$.
For $\eta>0$, let
\begin{equation*}
  T_{\eta}f(x)=\int_{|x-y|\ge\eta} K(x,y)f(y)\,dy.
\end{equation*}
Then the integral defining $T_{\eta}f(x)$ is convergent 
whenever $f\in L^p_{\comp}(\R^n)$ with $p\in(1,\infty)$.
We assume that, for all $p\in(1,\infty)$, there exists a positive constant $C_p$
such that for all $\eta>0$ and $f\in L^p_{\comp}(\R^n)$,
\begin{equation*} 
  \|T_{\eta}f\|_p \le C_p\|f\|_p, 
\end{equation*}
and that
\begin{equation*}
 \lim_{\eta\to0}T_{\eta}f=Tf
\end{equation*}
exists in $L^p(\R^n)$.
By this assumption, 
the operator $T$ can be extended to a continuous linear operator on $L^p(\R^n)$.
We shall say the operator $T$ satisfying the above conditions 
is a singular integral operator of type $\kappa$.
For example, Riesz transforms $R_j$, $j=1,\cdots,n$,
are singular integral operators of type $1$.

To define $T$ for Campanato spaces, 
we first define the modified version of $T_{\eta}$ as follows:
\begin{equation*}\label{tildeT}
     {\tT}_{\eta}f(x)
       =\int_{|x-y|\ge\eta} 
          \big[K(x,y)-K(0,y)(1-\chi_1(y))\big]f(y) 
                 \,dy.
\end{equation*}
Then, for $f\in\cL_{p,\lambda}(\R^n)$, $p\in(1,\infty)$, $\lambda\in[-n/p,1)$,
we can show that the integral in the definition above 
converges absolutely for all $x$ 
and that
${\tilde T}_{\eta}f$ converges 
in $L^{p}(Q)$ as $\eta\to0$ for each $Q$ 
(see the proof of \cite[Theorem~4.1]{Nakai2010RMC}).
We denote the limit by ${\tilde T}f$. 

\begin{rem}\label{rem:SI}
If $Tf$ is well defined, then ${\tT}f$ is also well defined and 
$Tf-{\tT}f$ is a constant function. Furthermore, for the constant function 
$1$, $T1$ is undefined, while ${\tT}1=0$.
See \cite[Remark~10]{KoMaNaSa2013RMC} for details.
\end{rem} 

The following results are known.

\begin{thm}[\cite{Nakai2010RMC,Peetre1966}]\label{thm:SI:Cam}
Let $T$ be a singular integral operator of type $\kappa\in(0,1]$.
Let $p\in(1,\infty)$.
If $\lambda\in[-n/p,\kappa)$,
then $\tT$ can be extended to a bounded operator 
on $\cL_{p,\lambda}(\R^n)$. 
Moreover, if $\lambda\in[0,\kappa)$, then
$\tT$ can be also extended to a bounded operator 
on $\cL_{1,\lambda}(\R^n)$. 
\end{thm}

\begin{thm}[\cite{KoMaNaSa2013RMC}]\label{thm:SI:Bs(Cam)}
Let $T$ be a singular integral operator of type $\kappa\in(0,1]$.
Let $\sigma\in[0,\infty)$ and $p\in(1,\infty)$.
If $-n/p+\sigma<\kappa$ and if $\lambda\in[-n/p,\kappa-\sigma)$,
then $\tT$ can be extended to a bounded operator 
on $B_{\sigma}(\cL_{p,\lambda})(\R^n)$
and $\dB_{\sigma}(\cL_{p,\lambda})(\R^n)$.
Moreover, if $\sigma<\kappa$ and if $\lambda\in[0,\kappa-\sigma)$, then
$\tT$ can be also extended to a bounded operator 
on $B_{\sigma}(\cL_{1,\lambda})(\R^n)$
and $\dB_{\sigma}(\cL_{1,\lambda})(\R^n)$. 
\end{thm}

Using Theorem~\ref{thm:SI:Bs(Cam)} and Example~\ref{exmp:MC3},
we have the following:

\begin{thm}\label{thm:SI:Bwu(Cam)}
Let $T$ be a singular integral operator of type $\kappa\in(0,1]$.
Let $p\in(1,\infty)$,
$u\in(0,\infty]$, $\Theta\in\iTheta$,
and let
\begin{equation}\label{tT w}
 w(r)=r^{-\sigma}\Theta(r^{\tau}),  \quad
 \sigma,\tau\in(0,\infty) \ \text{with} \ \sigma>\tau.
\end{equation}
If $-n/p+\sigma<\kappa$ and if $\lambda\in[-n/p,\kappa-\sigma)$,
then $\tT$ can be extended to a bounded operator 
on $B_w^u(\cL_{p,\lambda})(\R^n)$
and $\dB_w^u(\cL_{p,\lambda})(\R^n)$.
Moreover, if $\sigma<\kappa$ and if $\lambda\in[0,\kappa-\sigma)$, then
$\tT$ can be also extended to a bounded operator 
on $B_w^u(\cL_{1,\lambda})(\R^n)$
and $\dB_w^u(\cL_{1,\lambda})(\R^n)$. 
\end{thm}

Let $\lambda=0$ in Theorem~\ref{thm:SI:Bwu(Cam)} we have the following.

\begin{cor}\label{cor:SI:Bwu(BMO)}
Let $T$ be a singular integral operator of type $\kappa\in(0,1]$.
Let $u\in(0,\infty]$, $\Theta\in\iTheta$,
and define $w$ by \eqref{tT w}.
If $\sigma<\kappa$,
then $\tT$ can be extended to a bounded operator 
on $B_w^u(\BMO)(\R^n)$ and $\dB_w^u(\BMO)(\R^n)$. 
\end{cor}

By Theorem~\ref{thm:Lip} we have the following:

\begin{cor}\label{cor:SI:Bwu(Lip)}
Let $T$ be a singular integral operator of type $\kappa\in(0,1]$.
Let 
$u\in(0,\infty]$, $\Theta\in\iTheta$,
and define $w$ by \eqref{tT w}.
If $\sigma<\sigma+\alpha<\kappa$,
then $\tT$ can be extended to a bounded operator 
on $B_w^u(\Lip_{\alpha})(\R^n)$ and $\dB_w^u(\Lip_{\alpha})(\R^n)$. 
\end{cor}

\subsection{Modified fractional integral operators}\label{ss:tFI}

To define fractional integral operators on Campanato spaces
we define the modified version of $I_{\alpha}$, $\alpha\in(0,n)$, 
as follows;
\begin{equation*}
     \tI_{\alpha}f(x)
       =\int_{\R^n} f(y) 
                 \left(\frac {1}{|x-y|^{n-\alpha}} 
                       -\frac {1-\chi_{1}(y)}{|y|^{n-\alpha}}
                       \right)
                 \,dy.
\end{equation*}
If $I_{\alpha}f$ is well defined, 
then $\tI_{\alpha}f$ is also well defined and 
$I_{\alpha}f-\tI_{\alpha}f$ is a constant function. 
For the constant function $1$, $I_{\alpha}1\equiv\infty$, while
$\tI_{\alpha}1$ is well defined and also a constant function,
see \cite[Remark~2.1]{MatsuNakai2011} for example.

The following is known:
\begin{thm}[\cite{MatsuNakai2011}]\label{thm:FI Bs(Cam)}
Let $\alpha\in(0,1)$, $p,q\in[1,\infty)$, 
$\lambda\in[-n/p,1)$, $\mu\in[-n/q,1)$, $\sigma\in[0,\infty)$,
$\lambda+\alpha=\mu$
and $\sigma+\lambda+\alpha<1$.
Assume that $p$ and $q$ satisfy one of the following conditions:
\begin{enumerate}
\item 
$p=1$ and $1\le q<n/(n-\alpha)$;
\item
$1<p<n/\alpha$ and $1\le q\le pn/(n-p\alpha)$;
\item
$n/\alpha\le p<\infty$ and $1\le q<\infty$ {\rm (}in this case, $0\le\mu<1${\rm )}.
\end{enumerate}
Then $\tI_{\alpha}$ is bounded 
from $B_{\sigma}(\cL_{p,\lambda})(\R^n)$ 
to $B_{\sigma}(\cL_{q,\mu})(\R^n)$
and 
from $\dB_{\sigma}(\cL_{p,\lambda})(\R^n)$ 
to $\dB_{\sigma}(\cL_{q,\mu})(\R^n)$. 
\end{thm}

Using Theorem~\ref{thm:FI Bs(Cam)} and Example~\ref{exmp:MC3},
we have the following:

\begin{thm}\label{thm:FI Bwu(Cam)}
Let $\alpha\in(0,1)$, $p,q\in[1,\infty)$, 
$\lambda\in[-n/p,1)$, $\mu\in[-n/q,1)$ 
and $\lambda+\alpha=\mu$.
Let
$u\in(0,\infty]$, $\Theta\in\iTheta$,
and let
\begin{equation}\label{tI w}
 w(r)=r^{-\sigma}\Theta(r^{\tau}),  \quad
 \sigma,\tau\in(0,\infty) \ \text{with} \ \sigma>\tau.
\end{equation}
Assume that $\sigma+\lambda+\alpha<1$.
Assume also that $p$ and $q$ satisfy one of the conditions
{\rm(i)}, {\rm(ii)} and {\rm(iii)} in Theorem~\ref{thm:FI Bs(Cam)}.
Then $\tI_{\alpha}$ is bounded 
from $B_w^u(\cL_{p,\lambda})(\R^n)$ 
to $B_w^u(\cL_{q,\mu})(\R^n)$
and 
from $\dB_w^u(\cL_{p,\lambda})(\R^n)$ 
to $\dB_w^u(\cL_{q,\mu})(\R^n)$. 
\end{thm}

If $\lambda=0$, then $\cL_{p,\lambda}=\BMO$.
If $0<\lambda<1$, then $\cL_{p,\lambda}=\Lip_{\lambda}$.
Therefore, we have the following:

\begin{cor} \label{cor:Lip}
Let $\alpha,\beta,\gamma\in(0,1)$
and $\alpha+\beta=\gamma$.
Let
$u\in(0,\infty]$, $\Theta\in\iTheta$,
and define $w$ by \eqref{tI w} with $\alpha+\beta+\sigma<1$.
Then $\tI_{\alpha}$ is bounded 
from $B_w^u(\BMO)(\R^n)$ to $B_w^u(\Lip_{\alpha})(\R^n)$, 
from $B_w^u(\Lip_{\beta})(\R^n)$ to $B_w^u(\Lip_{\gamma})(\R^n)$,
from $\dB_w^u(\BMO)(\R^n)$ to $\dB_w^u(\Lip_{\alpha})(\R^n)$ 
and
from $\dB_w^u(\Lip_{\beta})(\R^n)$ to $\dB_w^u(\Lip_{\gamma})(\R^n)$. 
\end{cor}

\subsection{Vector-valued boundedness}\label{ss:Vector}

In this section we state the vector-valued inequalities for 
$B_w^u(L_{p,\lambda})(\R^n)$ and $\dB_w^u(L_{p,\lambda})(\R^n)$.

\begin{defn}\label{defn:E(Br)-vector}
Let $\dmn=\R^n$ or $Q_r$ with $r>0$.
Let $p\in[1,\infty)$, $\lambda\in\R$ and $v\in(0,\infty]$.
For
\begin{equation*}
 E=L^p,\ WL^p,\ L_{p,\lambda} \ \text{or}\ WL_{p,\lambda},
\end{equation*}
let $E(\ell^v)(\dmn)$ 
be the sets of all sequences
of functions $\sqc{f}$ such that the following 
functional is finite:
\begin{equation*}
 \|\sqc{f}\|_{E(\ell^v)(\dmn)}
  =\left\|\left(\sum_{j=1}^\infty |f_j|^v\right)^{1/v}\right\|_{E(\dmn)},
\end{equation*}
where we use the obvious modification when $v=\infty$.
\end{defn} 

Then
$\{(E(\ell^v)(Q_r),\|\cdot\|_{E(\ell^v)(Q_r)})\}_{0<r<\infty}$
has the restriction and decomposition properties
for 
$E=L^p,\ WL^p,\ L_{p,\lambda}\ \text{or}\ WL_{p,\lambda}$,
since
\begin{equation*}
 \left(\sum_{j=1}^\infty |f_j|_{Q_r}|^v\right)^{1/v}
 =
 \left(\sum_{j=1}^\infty |f_j|^v\right)^{1/v}\bigg|_{Q_r}
 \quad\text{and}\quad
 \left(\sum_{j=1}^\infty |f_j\chi_r|^v\right)^{1/v}
 =
 \left(\sum_{j=1}^\infty |f_j|^v\right)^{1/v}\chi_r.
\end{equation*}

\begin{defn}\label{defn:Bs(E)-vector}
Let $p\in[1,\infty)$, $\lambda\in\R$, $u,v\in(0,\infty]$
and $w\in\cW^{u}$.
For 
\begin{equation*}
 E=L^p,\ WL^p,\ L_{p,\lambda}\ \text{or}\ WL_{p,\lambda},
\end{equation*}
let
$B_w^u(E(\ell^v))(\R^n)$ and $\dB_w^u(E(\ell^v))(\R^n)$ 
be the sets of all sequences $\sqc{f}$, $f_j\in E_Q(\R^n)$, such that
$\|\sqc{f}\|_{B_w^u(E(\ell^v))}<\infty$ 
and 
$\|\sqc{f}\|_{\dB_w^u(E(\ell^v))}<\infty$, 
respectively,
where
\begin{align*}
 \|\sqc{f}\|_{B_w^u(E(\ell^v))}
 &=\left\|w(r)\|\sqc{f}\|_{E(\ell^v)(Q_r)}\right\|_{L^u([1,\infty),dr/r)}, 
\\
 \|\sqc{f}\|_{\dB_w^u(E(\ell^v))}
 &=\left\|w(r)\|\sqc{f}\|_{E(\ell^v)(Q_r)}\right\|_{L^u((0,\infty),dr/r)}.
\end{align*}
\end{defn}

We consider sublinear operators $T$ as in Subsection~\ref{ss:SI} 
on vector-valued function spaces, that is,
\begin{equation*}
 T:\sqc{f}\mapsto\sqc{Tf}.
\end{equation*}
Then the following is an extension of Theorem \ref{thm:T:Bs(Mor)} 
to the vector-valued version.

\begin{thm}[\cite{KoMaNaSa2013RMC}]\label{thm:T:Bs(Mor)-vector}
Suppose that the parameters 
$\sigma,p,q,\lambda,\mu$ and $v$
satisfy
\begin{equation*}\label{T-parameter}
 \sigma\in[0,\infty), 
 p,q\in[1,\infty), \,
 \lambda\in[-n/p,0), \, \mu\in[-n/q,0) \ \text{and}\
 v\in(1,\infty].
\end{equation*}
Let $T$ be a sublinear operator defined on $L^1_{\comp}(\R^n)$ and
satisfy \eqref{subdiff} and 
\eqref{size} for some $\alpha\in[0,n)$ 
and $\Omega\in L^{\tp}(S^{n-1})$ with $\tp\in[1,\infty]$.
Assume one of the following conditions:
\begin{enumerate}
\item $\mu=\lambda+\alpha$, $\tp\ge p'$ and $\sigma+\lambda+\alpha<0$,
\item $\mu=\lambda+\alpha$, $\tp\ge q$ and $\sigma+\lambda+n/\tp+\alpha<0$.
\end{enumerate}
If $T$ can be extended to a bounded operator 
from $L_{p,\lambda}(\ell^v)(\R^n)$ to $L_{q,\mu}(\ell^v)(\R^n)$ 
or to $WL_{q,\mu}(\ell^v)(\R^n)$,
then $T$ can be further extended to a bounded operator 
from $B_{\sigma}(L_{p,\lambda}(\ell^v))(\R^n)$ 
to $B_{\sigma}(L_{q,\mu}(\ell^v))(\R^n)$
or to $B_{\sigma}(WL_{q,\mu}(\ell^v))(\R^n)$, respectively.
That is,
\begin{equation*}
 \left\|\left(\sum_{j=1}^\infty |Tf_j|^v\right)^{1/v}\right\|_{B_{\sigma}(L_{q,\mu})}
 \le C
 \left\|\left(\sum_{j=1}^\infty |f_j|^v\right)^{1/v}\right\|_{B_{\sigma}(L_{p,\lambda})},
 \quad
 \text{if $p\in(1,\infty)$},
\end{equation*}
and
\begin{equation*}
 \left\|\left(\sum_{j=1}^\infty |Tf_j|^v\right)^{1/v}\right\|_{B_{\sigma}(WL_{q,\mu})}
 \le C
 \left\|\left(\sum_{j=1}^\infty |f_j|^v\right)^{1/v}\right\|_{B_{\sigma}(L_{1,\lambda})},
 \quad
 \text{if $p=1$},
\end{equation*}
where we use the obvious modification when $v=\infty$.
The same conclusion holds 
for $\dB_{\sigma}(L_{p,\lambda}(\ell^v))(\R^n)$.
\end{thm}

Using Theorem~\ref{thm:T:Bs(Mor)-vector} and Example~\ref{exmp:MC3},
we have the following:

\begin{thm}\label{thm:T:Bwu(Mor)-vector}
Let
\begin{equation}\label{v-T-w}
\begin{cases}
 p,q\in[1,\infty), \,
 \lambda\in[-n/p,0), \, \mu\in[-n/q,0), v\in(1,\infty],
 u\in(0,\infty], 
\\
 w(r)=r^{-\sigma}\Theta(r^{\tau}),  \ \Theta\in\iTheta, \ \text{and} \
 \sigma,\tau\in(0,\infty) \ \text{with} \ \sigma>\tau.
\end{cases}
\end{equation}
Let $T$ be a sublinear operator defined on $L^1_{\comp}(\R^n)$ and
satisfy \eqref{subdiff} and 
\eqref{size} for some $\alpha\in[0,n)$ 
and $\Omega\in L^{\tp}(S^{n-1})$ with $\tp\in[1,\infty]$.
Assume one of the following conditions:
\begin{enumerate}
\item $\mu=\lambda+\alpha$, $\tp\ge p'$ and $\sigma+\lambda+\alpha<0$,
\item $\mu=\lambda+\alpha$, $\tp\ge q$ and $\sigma+\lambda+n/\tp+\alpha<0$.
\end{enumerate}
If $T$ can be extended to a bounded operator 
from $L_{p,\lambda}(\ell^v)(\R^n)$ to $L_{q,\mu}(\ell^v)(\R^n)$ 
or to $WL_{q,\mu}(\ell^v)(\R^n)$,
then $T$ can be further extended to a bounded operator 
from $B_w^u(L_{p,\lambda}(\ell^v))(\R^n)$ 
to $B_w^u(L_{q,\mu}(\ell^v))(\R^n)$
or to $B_w^u(WL_{q,\mu}(\ell^v))(\R^n)$, respectively.
That is,
\begin{equation*}
 \left\|\left(\sum_{j=1}^\infty |Tf_j|^v\right)^{1/v}\right\|_{B_w^u(L_{q,\mu})}
 \le C
 \left\|\left(\sum_{j=1}^\infty |f_j|^v\right)^{1/v}\right\|_{B_w^u(L_{p,\lambda})},
 \quad
 \text{if $p\in(1,\infty)$},
\end{equation*}
and
\begin{equation*}
 \left\|\left(\sum_{j=1}^\infty |Tf_j|^v\right)^{1/v}\right\|_{B_w^u(WL_{q,\mu})}
 \le C
 \left\|\left(\sum_{j=1}^\infty |f_j|^v\right)^{1/v}\right\|_{B_w^u(L_{1,\lambda})},
 \quad
 \text{if $p=1$},
\end{equation*}
where we use the obvious modification when $v=\infty$.
The same conclusion holds 
for $\dB_w^u(L_{p,\lambda}(\ell^v))(\R^n)$.
\end{thm}

\begin{cor}\label{cor:SI:Bwu(Mor)-vector}
Let $p,\lambda,u,v,\Theta,\sigma,\tau$ and $w$ be as in \eqref{v-T-w}.
Assume that $\sigma+\lambda<0$. 
If a singular integral operator $T$ is bounded 
on $L^p(\ell^v)(\R^n)$ with $p\in(1,\infty)$, 
then $T$ can be extended to a bounded operator 
on $B_w^u(L_{p,\lambda}(\ell^v))(\R^n)$.
If $T$ is bounded from $L^1(\ell^v)(\R^n)$ to $WL^1(\ell^v)(\R^n)$,
then $T$ can be extended to a bounded operator 
from $B_w^u(L_{1,\lambda}(\ell^v))(\R^n)$ 
to $B_w^u(WL_{1,\lambda}(\ell^v))(\R^n)$.
The same conclusion holds for $\dB_w^u(L_{p,\lambda}(\ell^v))(\R^n)$.
\end{cor}

\begin{cor}\label{cor:FI:Bwu(Mor)-vector}
Let $\alpha\in(0,n)$, and 
let $p,q,\lambda,\mu,u,v,\Theta,\sigma,\tau$ and $w$ be as in \eqref{v-T-w}.
Assume that $\mu=\lambda+\alpha$, $q\le(\lambda/\mu)p$
and $\sigma+\lambda+\alpha<0$.
Then fractional integral operators $I_{\alpha}$ are bounded 
from $B_w^u(L_{p,\lambda}(\ell^v))(\R^n)$ 
to $B_w^u(L_{q,\mu}(\ell^v))(\R^n)$ if $p\in(1,\infty)$, 
and from $B_w^u(L_{1,\lambda}(\ell^v))(\R^n)$ 
to $B_w^u(WL_{q,\mu}(\ell^v))(\R^n)$ if $p=1$.
The same conclusion holds for $\dB_w^u(L_{p,\lambda}(\ell^v))(\R^n)$.
\end{cor}

On fractional maximal operators $M_{\alpha}$, $\alpha\in[0,n)$,
in the case $\sigma+\lambda+\alpha=0$,
Theorem \ref{thm:M:Bs(Mor)} can be extended to 
the vector-valued version in only the case $v=\infty$,
see \cite[Theorem~15 and Remark~14]{KoMaNaSa2013RMC}.

\begin{cor}\label{cor:M:Bwu(Mor)-vector}
Let $\alpha\in[0,n)$, and 
let $p,q,\lambda,\mu,u,\Theta,\sigma,\tau$ and $w$ be as in \eqref{v-T-w}.
Assume that $\mu=\lambda+\alpha$ and $q\le(\lambda/\mu)p$.
Assume also one of the following conditions.
\begin{enumerate}
\item $\sigma+\lambda+\alpha<0$ and $v\in(1,\infty]$,
\item $\sigma+\lambda+\alpha=0$ and $v=\infty$.
\end{enumerate}
Then the operator $M_{\alpha}$ can be extended to a bounded operator 
from $B_w^u(L_{p,\lambda}(\ell^v))(\R^n)$ 
to $B_w^u(L_{q,\mu}(\ell^v))(\R^n)$ 
if $p\in(1,\infty)$,
and from $B_w^u(L_{1,\lambda}(\ell^v))(\R^n)$ 
to $B_w^u(WL_{q,\mu}(\ell^v))(\R^n)$ 
if $p=1$.
The same conclusion holds for $\dB_w^u(L_{p,\lambda}(\ell^v))(\R^n)$.
\end{cor}

\section*{Acknowledgments}
The authors were partially supported by Grant-in-Aid 
for Scientific Research (C), No.~24540159, 
Japan Society for the Promotion of Science.


\end{document}